\documentclass[11pt]{article}

\usepackage{mathtools, amsmath, amsthm, amsfonts,amssymb, mathrsfs}
\usepackage{enumitem, graphicx, colortbl, tikz, bbm, dsfont, wrapfig}
\usepackage[utf8]{inputenc}
\usepackage{esint, subfig,float, graphics,epsfig,psfrag, color,soul}
\usepackage[T1]{fontenc}
\usepackage[active]{srcltx} 
\usepackage{comment}
\usepackage{microtype,stmaryrd}

\DisableLigatures{encoding = *, family = * }

\newtheorem{theorem}{Theorem}[section]

\newtheorem{lemma}[theorem]{Lemma}

\theoremstyle{definition}
\newtheorem{definition}[theorem]{Definition}
\newtheorem{remark}{Remark}

\DeclareMathOperator\diam{diam}

\DeclareMathOperator\sign{sign}

\DeclareMathOperator\Ker{Ker}
\DeclareMathOperator\hess{Hess}

\def\N{\mathbb{N}}

\def\R{\mathbb{R}}

\def\T{\mathbb{T}}
\def\S{\mathbb{S}}

\let\O=\Omega
\let\ve=\varepsilon

\let\t=\tilde

\let\.=\cdot
\let\0=\emptyset

\let\mc=\mathcal

\let\O=\Omega

\let\t=\tilde

\let\.=\cdot
\let\0=\emptyset

\let\mc=\mathcal

\def\1{\mathbbm{1}}

\def\t f{\tilde{f}}

\def\d{\,\mathrm{d}}
\def\ve{\varepsilon}

\newcommand\blfootnote[1]{%
	\begingroup
	\renewcommand\thefootnote{}\footnote{#1}%
	\addtocounter{footnote}{-1}%
	\endgroup
}

\newenvironment{formula}[1]{\begin{equation}\label{#1}}
	{\end{equation}\noindent}

\def\Fi#1{\begin{formula}{#1}}
	\def\Ff{\end{formula}\noindent}

\usepackage[colorlinks=true,linkcolor=black,citecolor=blue,urlcolor=blue,breaklinks]{hyperref}

\usepackage{hyperref}
\usepackage{breakurl}
\usepackage[square,sort,comma,numbers]{natbib}
\usepackage{url}
\definecolor{lpink}{rgb}{0.96, 0.76, 0.76}
\definecolor{dpink}{rgb}{0.97, 0.51, 0.47}
\definecolor{sky}{rgb}{0.53, 0.81, 0.92}
\definecolor{salmon}{rgb}{1.0, 0.55, 0.41}
\definecolor{orman}{rgb}{0.24, 0.7, 0.44}
\definecolor{aciksari}{rgb}{0.91, 0.84, 0.42}
\definecolor{dgrey}{rgb}{0.52, 0.52, 0.51}
\definecolor{ao}{rgb}{0.0, 0.5, 0.0}

\def\R{\mathbb{R}}

\def\d{\,\mathrm{d}}
\def\O{\Omega}

\def\p{\,\partial}

\renewcommand{\thefootnote}{\roman{footnote}}

\hypersetup{pdftitle={HarrisReview}}
\hypersetup{pdfauthor={Havva Yoldaş }}

\author{Havva Yolda\c{s}\footnote{Delft Institute of Applied Mathematics, Faculty of Electrical Engineering, Mathematics and Computer Science, Delft University of Technology, Mekelweg 4, 2628CD Delft, Netherlands. h.yoldas@tudelft.nl}}

\title{On quantitative hypocoercivity estimates based on Harris-type theorems}

\makeindex

\begin{document}
	
	\maketitle
	
	\begin{abstract}
		\noindent 
		
		This review concerns recent results on the quantitative study of convergence towards the stationary state for spatially inhomogeneous kinetic equations. We focus on analytical results obtained by means of certain probabilistic techniques from the ergodic theory of Markov processes. These techniques are sometimes referred to as Harris-type theorems. They provide constructive proofs for convergence results in the $L^1$ (or total variation) setting for a large class of initial data. The convergence rates can be made explicit (both for geometric and sub-geometric rates) by tracking the constants appearing in the hypotheses. Harris-type theorems are particularly well-adapted for equations exhibiting non-explicit and non-equilibrium steady states since they do not require prior information on the existence of stationary states. This allows for significant improvements of some already-existing results by relaxing assumptions and providing explicit convergence rates. We aim to present Harris-type theorems, providing a guideline on how to apply these techniques to the kinetic equations at hand. We discuss recent quantitative results obtained for kinetic equations in gas theory and mathematical biology, giving some perspectives on potential extensions to nonlinear equations.


	\end{abstract}
	
		\tableofcontents
	
	\section{Introduction}
	\label{sec:introduction}
	
	We are interested in the long-time behaviour of kinetic equations that can be written as
	\begin{equation}\label{eq:kinetic}
		\begin{alignedat}{2}
			\partial_t f(t,z)+ \mathcal T [f] (t,z)&= \mathcal C [f](t,z), \quad &&z \in \O\times \mc V, \, t >0, \\
			f(0,z) &= f_0 (z),  \quad  \quad  &&z \in \O\times \mc V,
		\end{alignedat}
	\end{equation} where $f(t,z) \geq 0$ is the probability density function of particles (or, depending on the context, gas molecules, bacteria cells, neurons etc.) at time $t \geq 0$ in the phase space, i.e. $z:=(x,v) \in \O \times \mc V$. In general, $x$ and $v$ represent the position and the velocity variables respectively\footnote{Except for the FiztHugh-Nagumo equation (see Section \ref{sec:FitzHugh-Nagumo}).}. 
	
	The operator $\mc T$ stands for the \emph{transport} part and it is either $\mc T [f] = v \cdot \nabla_x f$ describing the free transport or it takes the form $\mc T [f] = v \cdot \nabla_x f - \nabla_x \Phi(x) \cdot \nabla_v f$\footnote{$\nabla_x \cdot$ and $\nabla_v \cdot$ are the divergence operators in space and velocity variables respectively.} in the presence of a confining potential $\Phi$. In the latter case, an external force is exerted on the particles via the operator $-\nabla_x \Phi (x) \cdot \nabla_v$. 
	
	The operator $\mc C$ describes the \emph{collisions} (or interactions) between particles (or in some cases tumbling or firing process) and acts only on the velocity variable $v$. The definition of $\mc C$ specifies the kinetic equation we are interested in.

	The study of the long-time behaviour of kinetic equations involves proving that the solutions of \eqref{eq:kinetic} converge towards global equilibrium state as $t \to + \infty$ and estimating the rate of convergence. Both of these are well-known and important problems in kinetic theory (and in general for partial differential equations (PDEs)). The typical dynamics governed by \eqref{eq:kinetic} is that the transport part drives the process away from local equilibria unless the system is already in global equilibrium. It is conservative and it takes place only in the space variable $x$. On the other hand, the dissipation happens only on the velocity variable $v$ via the operator $\mc C$ and it drives the process towards local equilibrium. Then the transport term \emph{mixes} the dissipation into the space $x$ and leads to convergence to the global equilibrium. One needs to find a way to quantify this mixing effect to study the long-time behaviour of the equation. The theory of \emph{hypocoercivity} was developed specifically to quantify the effect of the transport (non-dissipative) part on the dissipative part (see \cite{V09, HN04,H06}) for kinetic equations.

	More precisely, hypocoercivity consists in finding a positive constant $C$ and a positive function $\beta(t)$ with $\beta(t) \to 0$ as $t \to +\infty$ such that
	\begin{align} \label{eq:hypocoercivity1}
		\|f (t,z) - f_\infty(z)\|_* \leq C \beta(t)\|f(0,z)-f_\infty(z)\|_*,
	\end{align} where $f$ solves \eqref{eq:kinetic} at time $t$ and $f_\infty$ is the stationary solution of \eqref{eq:kinetic} (it is a time independent solution of \eqref{eq:kinetic}, i.e. it satisfies that $0 =\p_t f_\infty = (\mc C - \mc T)[f_\infty]$). If $\beta(t) =  e^{-\lambda t} $ for some $\lambda >0$, then the convergence rate is exponential (or geometric). In this case, if one can show that such an estimate holds for \eqref{eq:kinetic}, then \eqref{eq:kinetic} is said to be \emph{hypocoercive} in the distance $\| \cdot \|_*$. Moreover depending on the form of $\beta(t)$ the convergence might be slower than exponential (sub-exponential or sub-geometric), e.g. in the case where $\beta(t)$ is an inverse power of a polynomial function.
	
	In this review, we consider kinetic equations whose dynamics are driven by continuous-time Markov processes and we are interested in showing an inequality like \eqref{eq:hypocoercivity1} for these equations. Moreover we are interested in  quantifying $C$ and $\beta(t)$ in \eqref{eq:hypocoercivity1}. In probability theory, the study of the long-time behaviour of Markov processes is referred to as the \emph{ergodic theory of Markov processes}. There are two classical approaches to obtaining quantitative ergodicity estimates for continuous-time Markov processes (see \cite{BCG08}):
	\begin{enumerate}[label=(\roman*)]
		\item {\bf Poincaré-type inequalities} are based on obtaining integral bounds on the generator of the process.
		\item {\bf Harris-type theorems} are based on finding an appropriate Lyapunov function which satisfies an inequality on the generator of the process, and the inequality is valid in a \emph{small set}.
	\end{enumerate}

	In this review, our focus is on the results using the second approach. Harris-type theorems originated in 1940 (cf. \cite{D40}) in the study of irreducibility of Markov chains and they are commonly used among probabilists since then. In the last couple of decades, these ideas have been revisited. There are several works providing quantitative versions of these theorems and alternative proofs using only PDE techniques e.g. semigroup theory, rather than probabilistic arguments (see e.g. \cite{MT94, HM11, CM21}). These alternative proofs provided a concise relation between the ergodic theory of Markov processes and the spectral properties of operators defining these PDEs. In turn, these results became a subject of many recent works in the context of PDEs. Particularly they are used in the study of the asymptotic behaviour of some models in population dynamics (e.g., \cite{G18, BCGM19,CY19, BCG20, CGY21, Y19}), kinetic theory (e.g. \cite{B20, CCEY20, C20, C21, E18}) and extended to non-conservative cases, e.g. systems where the mass is not conserved (e.g. \cite{BCGM19, BCG20, CG20}).

	\paragraph{Motivation and aim.} This review aims at collecting recent results which provide quantitative hypocoercivity estimates for some well-known kinetic equations in physics and biology using Harris-type arguments. These techniques have been used efficiently  to show the existence of a stationary state and to obtain quantifiable convergence rates to the stationary state. Harris-type theorems have certain advantages over classical hypocoercivity techniques in some cases. We discuss some of them below.
	
	\begin{itemize}
		\item {\bf Non-explicit and non-equilibrium steady states.} The classical hypocoercivity results often involve a step where one should prove a Poincaré-type inequality in weighted $L^2/H^1$ norms where the weight function is the inverse of the invariant measure of the process, i.e. the stationary solution of the equation. There is a large class of kinetic equations describing the gas dynamics for which the equilibrium state is given by a Gibbs-like distribution, i.e. it can be written explicitly. However, kinetic equations arising in mathematical biology (or in applied sciences in general) typically have complex and non-explicit steady states. This creates some difficulties when trying to prove a Poincaré inequality. In \cite{DMS15}, the authors provided a general methodology to obtain hypocoercivity results for linear kinetic equations preserving mass. The condition given in \cite{DMS15} for $f_\infty$ to be a steady state of a kinetic equation \eqref{eq:kinetic} is that $f_\infty$ should belong to the kernels of both the transport and the collision operators, i.e. $f_\infty \in \{ \Ker (\mc T)\cap \Ker (\mc C)\}$. In Section \ref{sec:runtumble}, we look at the run and tumble equation which describes the movement of bacteria under the effect of a chemical stimulus. It is a mass-preserving, linear kinetic equation. However the stationary state of the run and tumble equation does not satisfy the condition given in \cite{DMS15}, thus it cannot be treated with the methods proposed there. Harris-type theorems do not require any information about the stationary state of the equation to be applied. They are well-adapted for situations where the stationary solutions can be complex and non-explicit. They also provide the existence of a stationary solution and the convergence towards it simultaneously. Moreover, there are many physical systems in which the external effects, such as boundary conditions, drive the system from equilibrium to a non-equilibrium state. The latter is not a Gibbs-like state and is often not explicit. In another example we consider in Section \ref{sec:nonlin_BGK}, Harris-type theorems are used efficiently to show the existence of a stationary state for a nonlinear BGK equation which exhibits non-equilibrium stationary states. 
		
		\item {\bf Initial data and functional space.} Concerning the class of initial data for a kinetic equation, the natural assumption to make in terms of physical relevance is that it should be a finite measure. However, most of the time more restrictive assumptions are made if we would like to use classical hypocoercivity techniques since the estimates we obtain are in the $L^2/H^1$ setting. For example, to study the linear kinetic Fokker-Planck equation, a usual assumption to make is $\int f^2/f_\infty < +\infty$. If we write the equation in terms of the variable $h:=f/f_\infty$ we can benefit from the fact that the collision term is self-adjoint in the space $L_{f_\infty}^2$\footnotemark. 	\footnotetext{$L_{f_\infty}^2$ stands for the weighted $L^2$ space with a weight $f_\infty$.}This assumption restricts us to obtain convergence to equilibrium in $L^1$ spaces (see e.g. \cite{V02} for a more detailed discussion and how to obtain convergence estimates in the $L^1$ setting for the linear kinetic Fokker-Planck equation with the classical techniques). On the other hand, Harris-type theorems are very suitable for obtaining convergence estimates in the $L^1$ (or total variation) setting and they do not require any type of assumption or restriction on the initial data. They are valid for a wider class of initial data, including Dirac measures with bad local regularities and initial data with slowly decaying tails. Moreover we also mention that there are other techniques coming from analysis (see e.g.\cite{GMM17, MK14, BS13, BS13_2} and references therein) and from probability (see e.g.  for \emph{coupling techniques} \cite{FGM12,FGM16, DGM21, BF22}) providing estimates in the $L^1$ setting. 
		
		\item {\bf Constructive proofs and quantifiable convergence rates.}
		If we are interested in the qualitative properties of the solutions especially to compare if the analytical results are physically relevant at least in a certain range of parameters, we look for explicit bounds on the convergence rates. In this regard, Harris-type theorems have an advantage over non-quantitative techniques like the Kre\u{\i}n-Rutman theorem. It is straightforward to obtain the convergence rates quantitatively by tracking the constants appearing in the hypotheses of Harris-type theorems. The proofs are all constructive. However, in some cases, it involves computing complicated integrals. There are also works using Harris-type arguments in a non-quantitative manner for kinetic equations, e.g. \cite{BL55, CELMM18}. In some cases, it is possible to obtain quantitative convergence rates with other hypocoercivity techniques as well but it is not forthright as in the case of Harris-type theorems.
		
		\item {\bf Kinetic models arising in applied sciences.} Even though the kinetic equations obtained as a mean-field limit of biological processes share a common form with the kinetic models in gas theory, they differ drastically from them in terms of confinement mechanism and the nature of their steady states. We already mentioned that these models exhibit non-explicit and sometimes non-equilibrium steady states. As for the confinement mechanism, the confinement in these models is not induced through an external force field. It is a consequence of the internal dynamics, e.g. bias in the tumbling direction towards the chemoattractant for the run and tumble equation and the internally generated noise as input for membrane potential for the FitzHugh-Nagumo equation (see Sections \ref{sec:runtumble} and \ref{sec:FitzHugh-Nagumo} for more details). Due to these differences, sometimes the classical hypocoercivity approaches which are initially developed for treating the kinetic models in gas dynamics, are not applicable for the kinetic models in applied sciences or these approaches provide results under restrictive assumptions. However, Harris-type theorems are well suited to treat these models. 
		
	\end{itemize}
	
	\paragraph{Lyapunov vs. Poincaré.} Finally, we would like to mention the work \cite{BCG08} where the authors provided a link between the aforementioned two approaches: Poincaré-type inequalities and Harris-type theorems. The existence of a Lyapunov function is a sufficient (also necessary for exponential convergence) condition to be able to use Harris-type theorems. On the other hand, a general condition for obtaining Poincaré-type inequalities is not known. The significance of  \cite{BCG08} comes from the fact that the authors formed a bridge between these two approaches through some new inequalities called \emph{Lyapunov–Poincaré inequalities}. They also improved some existing results by using those. 
	
	\paragraph{Organisation of the paper.}
	
	In Section \ref{sec:Harris}, we introduce some common concepts in measure theory, probability and introduce some notations. This is followed by the statements of Harris-type theorems. We skip their proofs (which can be found in various references we list there) and rather focus on the intuitions behind the hypotheses and the results by discussing how to verify them for a given PDE. In Section \ref{sec:hypocoercivity}, we show how these theorems are applied to various kinetic models to obtain quantitative estimates. Section \ref{sec:hypocoercivity} is divided into two parts: 
	\begin{itemize}
		\item Section \ref{sec:MathPhys} is dedicated to some kinetic equations coming from the kinetic theory of gases. After an introduction of the Boltzmann equation, which is one of the fundamental equations in kinetic theory, we present quantitative results on the linear BGK equation (Sec. \ref{sec:linearBGK}), the collisionless (Knudsen) gas with boundary conditions (Sec. \ref{sec:Knudsen}), the nonlinear BGK equation on an interval (Sec. \ref{sec:nonlin_BGK}),  the linear Boltzmann equation (Sec. \ref{sec:lin_Boltzmann}), the linear degenerate Boltzmann equation (Sec. \ref{sec:deg_Boltzmann}) and the linear kinetic Fokker-Planck equation (Sec. \ref{sec:FP}). 
		\item Section \ref{sec:MathBio} is dedicated to two examples of kinetic models arising in biological processes: the run and tumble equation for bacterial chemotaxis (Sec. \ref{sec:runtumble}) and the FitzHugh-Nagumo equation for interacting neuron cells (Sec. \ref{sec:FitzHugh-Nagumo}).
	\end{itemize} In each of these sections, we introduce the corresponding equation, mention some literature and present the results on the equation in the form of a theorem and give a strategy for their proofs skipping the technicalities. We remark that our focus is on the quantitative results using Harris-type arguments.
	The last section (Section \ref{sec:Discussion}) is dedicated to a final discussion and some perspectives.
	
	\section{Harris-type theorems}
	\label{sec:Harris}
	
	As far as this review is concerned, by \emph{Harris-type theorems}\footnote{We use the same terminology as in \cite{CM21}.}, we refer to some probabilistic results on the study of ergodic (long-time) behaviour of Markov processes. A way of describing a Markov process is through its transition probabilities which generate a semigroup of linear operators on appropriate function spaces. Studying the spectral properties of this semigroup is closely related to the ergodicity of the Markov process and it is an important topic of research for both partial differential equations and probability.
	
	The original ideas of the Harris-type theorems date back to the paper of Doeblin \cite{D40} where he showed that once its transition probabilities have a uniform lower bound, a Markov chain has a mixing property meaning that it will admit an invariant (or a stationary) measure. A detailed explanation of these ideas can be found in Chapter 2 of the book \cite{S14}. In \cite{H56}, Harris studied the conditions for the existence of a unique invariant measure for a Markov process. Later Meyn and Tweedie \cite{MT93III, MT94} extended these results by showing exponential (or geometric) convergence rates to a unique stationary measure for some Markov processes. In \cite{DFG09, FR05}, the authors adapted the framework of Meyn and Tweedie to show how to obtain sub-exponential (or subgeometric) convergence rates. In \cite{HM11}, Hairer and Mattingly gave an alternative proof of this result in appropriate mass transport distances (yielding the total variation setting) including the cases with sub-exponential convergence rates (see also \cite{BCG08}). Following the paper of Hairer and Mattingly, Harris-type theorems gained a lot of interest outside the probability community, in particular, they have been successfully used for determining the spectral gaps of integro-differential operators which describe Piecewise Deterministic Markov Processes (PDMPs). In this review, we particularly focus on some kinetic equations motivated by physical or biological processes (see Sections \ref{sec:MathPhys} and \ref{sec:MathBio} respectively). The significance of \cite{HM11} comes from the fact that the authors provided an efficient way of obtaining quantitative convergence rates to a unique equilibrium state once the hypotheses are verified quantitatively. Moreover, recently in \cite{CM21}, Cañizo and Mischler provided yet another proof relying only on semigroup and interpolation arguments essentially using similar ideas. 
	
	\, 
	
	Harris-type theorems provide geometric (i.e. exponential) or sub-geometric (e.g. polynomial) convergence rates to a unique equilibrium state for a Markov process relying on two hypotheses:
	
	\begin{enumerate}[label=(\roman*)]
		\item { \bf Strong-positivity or minorisation condition} refers to the \emph{irreducibility} or the \emph{uniform mixing} property of a Markov process. It consists in verifying that the probability of a Markov process transitioning from an arbitrary initial state to any other state in its state space is positive. If such a property holds uniformly on the whole state space then \emph{Doeblin's theorem} gives the existence of a unique stationary equilibrium state and exponential relaxation towards it. When this condition fails to hold uniformly but it holds in a subset of the state space, it is still possible to show the existence of a unique equilibrium state under the condition that the subset where a  mixing property holds is visited infinitely often. 
		
		\item { \bf Foster-Lyapunov condition} refers to a \emph{geometric drift} or a \emph{confinement} property of a Markov process. This condition ensures that the Markov process converges to the centre of the state space where a minorisation condition holds and it is verified by finding an appropriate Lyapunov function satisfying an inequality on the generator of the process. The speed of convergence,  i.e. whether it is geometric or sub-geometric, is determined depending on how strong this inequality is. 
	\end{enumerate}
	
	Before giving the statements of the theorems we introduce some notations.
	
	\paragraph{Notations} We are interested in the long-time behaviour of some continuous-time Markov processes taking values in locally compact separable metric spaces, e.g. Polish spaces. We consider a measurable space $(\Omega ,\Sigma )$ where $\Omega$ is a Polish space endowed with a probability measure, i.e. a measure of mass $1$. We denote $\mc M (\O)$ and $\mc P (\O)$ as the spaces of finite measures and probability measures on $\O$ respectively. 
	\begin{definition}[Total variation distance]
		The \emph{total variation distance} between two probability measures $\mu, \nu \in \mc P (\O)$ is defined as 
		\begin{align}
			\label{def:TVnorm}
			\|\mu (E) - \nu(E)\|_{TV} := \sup_{E \in \Sigma} |\mu (E) - \nu(E)|,
		\end{align} for every Borel set $E \in \Sigma$. 
	\end{definition}
	
	\begin{definition}[Absolutely continuous measures]
		A measure $\mu \in \mc M (\O)$ is \emph{absolutely continuous} if there exists a Lebesgue integrable function $f$ such that 
		\begin{align*}
			\mu (E) = \int_E f (\mathrm{d} z),
		\end{align*} for every Borel set $E \in  \Sigma$. Then $f$ is called the \emph{Radon-Nikodym derivative }$\d \mu / \d z$. 
	\end{definition}
	Next we give an elementary result establishing a connection between the total variation distance and the $L^1$ distance. Later we consider that these two distances are essentially the same. 
	\begin{lemma} \label{lem:TVtoL1}
		Let  $\mu, \nu \in \mc P (\O)$ and $f, g$ be the Radon-Nikodym derivatives of $\mu$ and $\nu$ respectively. Then 
		\begin{align*}
			\|\mu (E) - \nu(E)\|_{TV} = \frac{1}{2}\|f - g\|_{L^1}. \footnotemark
		\end{align*}
	\end{lemma}
	\footnotetext{We use the following notation: $\|\cdot \|_{L^1} = \|\cdot \|_{L^1(\O)}$ unless the set of integration is specified otherwise.}
	We skip the proof of this lemma. 
	
	We define the weighted $L^1$ distance (or weighted total variation distance) by 
	\begin{align*}
		\|f\|_\phi : = \int_\O \phi (z) |f(z)| (\mathrm{d} z),
	\end{align*} where $f$ is a measurable function and $\phi: \O \to [1, +\infty)$ is a measurable weight function. We define the space $\mc M_\phi := \{ f(\mathrm{d}z) \in \mc M (\O)  \mid f \phi \in L^1(\O)\}$ and we denote $\mc P_\phi (\O):= \mc M_\phi (\O) \cap \mc P$. Notice that when $\phi=1$ the norm corresponds to the $L^1$ or the total variation distance (on the set $\O$) and we use, with a slight abuse, the following notation: $\| \cdot \|_1 = \| \cdot \|_{L^1} = \| \cdot \|_{TV}$. Moreover for a given weight function $\phi$ we also define for $p \in [1, +\infty)$
	\begin{align*}
		L^p_\phi (\O):= \{f (\mathrm{d} z)\in \mc M (\O) \mid f \phi \in L^p (\O)\}, \quad \text{and} \quad \|f\|_{L^p_\phi} = \|f \phi \|_{L^p}.
	\end{align*} We denote $\mc P_\phi^p (\O) := L_\phi^p (\O) \cap \mc P$. 

For $n \in \N$, $1 \leq p \leq \infty$, we also define the Sobolev spaces $ W^{n,p}(\O)$ as
\begin{align*} 
	W^{n,p}(\O) : = \{ f \in L^p (\O) \mid D^\alpha f \in L^p (\O), \, 0 \leq | \alpha | \leq n \},
\end{align*} where, for every multi-index $\alpha$, $  D^\alpha f $ are the partial derivatives of $f$, i.e. $ D^\alpha f  = \dfrac{\p^{|\alpha|} f}{\p z_1^{\alpha_1} \cdots \p z_d^{\alpha_d} }$.

	A continuous-time Markov process is defined by a family of transition probability functions $P_t(z, E)$ for $t \geq 0$, $z \in \O$ where $E$ is a Borel set. Notice that $(t,z) \mapsto P_t(z, E)$ is a measurable function for all Borel sets $E$ and $E \mapsto P_t(z, E)$ is a probability measure on $\O$ for all $(t,z)$. The transition probability functions generate a linear semigroup $(S_t)_{t \geq0}$. The action of $S_t$ on finite measures and the action of its dual semigroup $S^*_t$ on measurable functions are given by 
	\begin{align*}
		S_t \mu (E) = \int_\O P_t(z, E) \mu (\mathrm{d} z) \quad \text{and} \quad S_t^* \phi (z) = \int_\O \phi (y) P_t (z, \d y),
	\end{align*} where $\phi : \O \to [0, + \infty)$ is a measurable function and $\mu \in \mc M(\O)$ is a finite measure on $\O$. 
	
	A Markov semigroup $S_t$ satisfies the following properties:
	\begin{itemize}
		\item $S_0 = \mbox{Id}$ or $P_o(z, \cdot) = \delta_z$ for all $z \in \O$.
		\item $S_{t+s} = S_t S_s$ for all $t,s \geq 0$ (semigroup property).
		\item $t \mapsto S_t \mu$ for all $\mu \in \mc M (\O)$ is continuous.
		\item $S_t^* \phi \geq 0$ if $\phi \geq 0$ for any $t \geq 0$ and any measurable function $\phi$ (positivity preserving).
		\item $\|S_t^* \phi \|_\infty = \|\phi \|_\infty$ for any measurable function $\phi$ (mass-preserving, contraction semigroup).
		\item $S_t^* \mathds{1}(z) = \mathds{1}(z) $ for any $z \in \O$.
	\end{itemize}
	A finite measure $\mu \in \mc M (\O)$ is an invariant/stationary measure for the Markov semigroup $S_t$ if $S_t \mu = \mu$ for all $t \geq 0$. Moreover if $f$ solves \eqref{eq:kinetic}, then $S_tf_0 (z)= f(t,z)=f_t$ represents the solution of  \eqref{eq:kinetic} with initial data $f_0$. 
	
	A Markov semigroup is called \emph{Feller} if it maps the space of bounded measurable functions to the space of continuous functions.
	
	\paragraph{Statements of the theorems.}  With these notations, now we give the statements of Harris-type theorems in the spirit of  \cite{HM11, DFG09, H21, CM21}. We skip the proofs and rather focus on the essence of these results and how to use them practically on some kinetic equations. We refer the reader to \cite{HM11, DFG09, H21, CM21} for the proofs. 
	
	Let us start with Doeblin's theorem. 
	\begin{theorem}[Doeblin's theorem]
		\label{thm:Doeblin} 
		Let $(S_t)_{t\geq 0}$ be a Markov semigroup defined on $\mathcal{M} (\Omega)$ satisfying the following condition:
		\begin{itemize}
			\item[] There exist a constant $\alpha \in (0,1)$, a probability measure $\eta$ and some time $\tau>0$ such that 
			\begin{align} \tag{\mbox{Doeblin's condition}} \label{Doeblin}
				S_{\tau} \mu \geq \alpha \eta, \quad \mbox{for all } \mu \in \mc P(\O).
			\end{align} 
		\end{itemize}
		Then $(S_t)_{t\geq 0}$ has a unique invariant probability measure $\mu_\infty$ and for any $\mu \in \mc P (\O)$
		\begin{align*}
			\|S_{t} \mu - \mu_\infty\|_{1} \leq C  e^{-\lambda t} \|\mu - \mu_\infty\|_{1}, \quad \mbox{for all } t \geq 0,
		\end{align*} where $C: = 1/ (1-\alpha)$, $\lambda : = - \log (1-\alpha )/ \tau$. 
	\end{theorem}

	Doeblin's theorem gives exponential convergence to a unique equilibrium once the Markov semigroup satisfies a positivity condition uniformly on the whole state space. However, it is often difficult to show such a uniform positivity condition, e.g. when the state space is unbounded. Harris's theorem extends this result to such cases via a Lyapunov type inequality on the transition probabilities governing the flow of the process.
	
	\begin {theorem} [Harris's theorem]	\label{thm:harris} Let $(S_t)_{t\geq 0}$ be a Markov semigroup defined in $\mathcal{M}(\O)$ satisfying the following two conditions:
	\begin{enumerate}[label=(\roman*)]
		\item There exist $\gamma \in (0,1)$, $K \geq 0$, a continuous function $\phi : \O\to [1, + \infty)$ and some time $\tau>0$ such that 
		\begin{align}
			\tag{Foster-Lyapunov condition} \label{Foster_Lyap}
			S_\tau^* \phi (z) \leq \gamma \phi (z) + K.
		\end{align} 
		\item There exist a constant $\alpha \in (0,1)$, a probability measure $\eta $ and some time $\tau>0$ such that
		\begin{align}
			\tag{minorisation condition} \label{Minorisation}
			S_{\tau} \mu \geq \alpha \eta, \quad \mbox{for all } \mu \in \mc P (\mc A),
		\end{align} where $\mc A: = \{z : \phi (z) \leq R\}$ and $R > 2K / (1 - \alpha)$. 
	\end{enumerate} Then $(S_t)_{t \geq 0}$ has a unique invariant probability measure $\mu_\infty$ and for any $\mu \in \mc P (\O)$ , there exist some constants $C>1$, $\lambda >0$ such that 
	\begin{align*}
		\|S_{t} \mu - \mu_\infty\|_{\phi} \leq C  e^{-\lambda t} \|\mu - \mu_\infty\|_{\phi}, \quad \mbox{for all } t \geq 0.
	\end{align*}
\end{theorem}

\begin{remark}
	In order to verify the \ref{Foster_Lyap} (especially for continuous-time Markov processes) it is convenient to consider the generator of the process since it is more accessible than the transition probability function in general. In fact, to show that the \ref{Foster_Lyap} holds what we require is a function $\phi(z)$ such that $\lim_{|z|\to +\infty} \phi(z) = +\infty$ satisfying
	\begin{align} \label{eq:FL2}
		\int \phi (z) f (\tau,\mathrm{d} z) \leq \gamma \int \phi (z) f_0(\mathrm{d }z) + K, 
	\end{align} where $f(\tau,z)$ solves an equation $\p_t f = \mc L f$ at time $\tau>0$ for any initial data $f_0$ which is a probability measure. Then the \ref{Foster_Lyap} is equivalent to the following inequality:
	\begin{align} \label{eq:FL3}
		\mc L^* \phi (z) \leq - \zeta \phi (z) + D,
	\end{align} for some positive constants $\zeta $ and $D$ where $\mc L^*$ is the formal adjoint of $\mc L$, or it is the \emph{forward operator} or the \emph{infinitesimal generator} of the Markov process defined via
	\begin{align*}
		\frac{\d}{\d t} S_t^* \psi \Big |_{t=0} = \mc L^* \psi,
	\end{align*} for all $\psi \in C_c^\infty (\O)$. We can verify the fact that \eqref{eq:FL3} implies the \ref{Foster_Lyap} with $\gamma = e^{-\zeta \tau }$ and $K = D/\zeta $ by an easy computation. First we notice that \eqref{eq:FL3} is nothing but the following: 
	\begin{align} \label{eq:FL4}
		\frac{\d}{\d t} \int \phi (z) f (t,\mathrm{d} z) \leq -\zeta  \int \phi (z) f (t,\mathrm{d} z) + D.
	\end{align} 
	We multiply $\int \phi (z) f (t,\mathrm{d} z)$ with $e^{\zeta  t}$ and take a derivative in $t$ to obtain, using \eqref{eq:FL4},
	\begin{align*}
		\frac{\d}{\d t} \left( e^{\zeta  t}\int \phi (z) f (t,\mathrm{d} z) \right) &= \zeta   e^{\zeta  t} \int \phi (z) f (t,\mathrm{d} z) + e^{\zeta  t}  \frac{\d}{\d t} \int \phi (z) f (t,\mathrm{d} z)  \\
		&\leq  \zeta  e^{\zeta  t} \int \phi (z) f (t,\mathrm{d} z) - \zeta   e^{\zeta  t} \int \phi (z) f (t,\mathrm{d} z)  + D e^{\zeta  t} \\
		&\leq D e^{\zeta  t}.
	\end{align*}
	Now we integrate this between $0$ and $\tau$ and we obtain
	\begin{align*}
		\int_0^\tau 	\frac{\d}{\d t} \left( e^{\zeta  t}\int \phi (z) f (t,\mathrm{d} z) \right)  \d t = e^{\zeta \tau }\int \phi (z) f (\tau ,\mathrm{d} z) - \int \phi (z) f _0(\mathrm{d} z) \leq D \int_0^\tau e^{\zeta t} \d t = \frac{D}{\zeta } (e^{\zeta \tau} -1)
	\end{align*} and this implies 
	\begin{align*}
		\int \phi (z) f (\tau ,\mathrm{d} z) \leq e^{-\zeta \tau } \int \phi (z) f _0(\mathrm{d} z)  + \frac{D}{\zeta } (1 - e^{-\zeta \tau}). 
	\end{align*} Moreover, we notice that the \ref{Foster_Lyap} is equal to
	\begin{align} \label{eq:FL5}
		\int \phi(z) f(\tau, \mathrm{d}z) \leq \gamma \int \phi(z) f_0(\mathrm{d}z)  + K. 
	\end{align}
	Then taking $\gamma = e^{-\zeta \tau }$ and $K = \frac{D}{\zeta} (1 - e^{-\zeta \tau}) \leq \frac{D}{\zeta}$ gives \eqref{eq:FL2}, equivalently the \ref{Foster_Lyap}.
\end{remark}
\begin{remark}
	The constants in Theorem \ref{thm:harris} can be computed explicitly. Setting $\gamma_0 \in [\gamma + 2K/R, 1)$, we choose $\beta = \alpha_0/K$ for any $\alpha_0 \in (0, \alpha)$ then we have $C := 1- \bar \alpha$  and $\lambda : = - \log (1- \bar \alpha)/ T$ where $\bar \alpha = \min \{ \alpha + \alpha_0, (2 + R\beta (2-\gamma_0))/ (2+ R\beta)\}$ (see Remark 3.10 in \cite{H21}).
\end{remark}
When it is not possible to find a Lyapunov function satisfying the \ref{Foster_Lyap}, but a weaker inequality, a version of Harris's theorem can still be used to obtain sub-geometric convergence rates. Here we state (and later use) a version which is found in \cite{H21, CM21} but for different versions we also refer to \cite{DFMS04} and particularly Theorem 5.2.c in \cite{DMT95}, Theorems 3.10 \& 3.12 in \cite{DFG09}, Theorem 1.2 in \cite{BCG08} and references therein.

\begin{theorem}[Harris's theorem for subgeometric convergence] \label{thm:subgeoHarris} Let $(S_t)_{t\geq 0}$ be a Markov semigroup defined in $\mathcal{M}(\O)$. Let $S_t$ be also Feller for all $t \geq 0$ satisfying the following two conditions:
	\begin{enumerate}[label=(\roman*)]
		\item There exist constants $\zeta >0$, $D \geq 0$, a continuous function $\phi : \O\to [1, + \infty)$ with pre-compact sublevel sets such that 
		\begin{align}
			\tag{weaker Foster-Lyapunov condition} \label{Foster_Lyap2}
			\mc L^* \phi (z) \leq -\zeta V (\phi) + D,
		\end{align} where $V: [1, + \infty)  \to [1, +\infty)$ is a strictly concave, positive and increasing function, and $\lim_{u\to + \infty} V^\prime (u)= 0$. 
		\item For every $R > 0$ there exist a constant $\alpha \in (0,1)$, a probability measure $\eta $ and some time $\tau>0$ such that
		\begin{align}
			\tag{sub-geometric minorisation condition} \label{Minorisation2}
			S_{\tau} \mu \geq \alpha \eta, \quad \mbox{for all } \mu \in \mc P ( \mc A),
		\end{align} where $\mc A : = \{z : \phi (z) \leq R\}$. 
	\end{enumerate} 
	Then $(S_t)_{t \geq 0}$ has a unique invariant probability measure $\mu_\infty$ satisfying
	\begin{align}
		\int V (\phi (z)) \mu_\infty (\mathrm{d} z) \leq D,
	\end{align} and there exists a constant $C$ such that for any $\mu \in \mc P (\O)$,
	\begin{align*}
		\|S_{t} \mu - \mu_\infty\|_{TV} \leq \frac{C\mu(\phi)}{H_V^{-1}(t)} + \frac{C}{(V \circ H_V^{-1})(t)},
	\end{align*} where $\mu (\phi) = \int \phi(z) \mu (\mathrm{d} z) $ and 
	\begin{align*}
		H_V = \int_1^t \frac{\d s}{V(s)}.
	\end{align*} 
\end{theorem}
For the proofs of these theorems we refer to \cite{H21, DFG09, CM21}.  We note that instead of the condition $\lim_{u\to + \infty} V^\prime (u)= 0$ in \cite{DFG09, H21}, a weaker condition, $V(u) \leq u $ for any $u \geq 1$, is assumed in \cite{CM21} (see Hypothesis 7).  The latter is weaker in the sense that it allows linear growth at infinity, whereas the former essentially requires for $V$ to be flat at infinity.

After presenting the statements of the Harris-type theorems we now look at their applications on some kinetic equations. 

\section{Quantitative hypocoercivity results}
\label{sec:hypocoercivity}
In this section, we present some examples of applications of the Harris-type theorems (see Section \ref{sec:Harris}) on some kinetic equations, of the type \eqref{eq:kinetic}, which are motivated by some problems in physics and biology. The common feature of these equations is that the underlying process is a continuous-time Markov process which is defined by a linear, positivity- and mass-preserving semigroup.

Let $t \mapsto f_t$ solve \eqref{eq:kinetic} at time $t \geq 1$. The mass conservation property gives that $\int f _0\d z  = \int f_t \d z$. This means that if we consider an initial data which is a probability measure (this will be the case unless stated otherwise) then it will remain so. This semigroup is also positivity preserving. As we already defined in Section \ref{sec:Harris}, a semigroup having these properties is called a \emph{Markov semigroup}. 

In the subsequent sections, we look at the quantitative hypocoercivity estimates for some kinetic equations appearing in the context of physics and biology respectively.

\subsection{Kinetic equations arising in mathematical physics}
\label{sec:MathPhys}

In 1872, Ludwig Boltzmann derived the so-called \emph{Boltzmann equation} which is considered the fundamental equation in the kinetic theory of gases (see \cite{B64,V02}). The Boltzmann equation describes the statistical behaviour of a thermodynamical system in a non-equilibrium state. It is derived for the evolution of a dilute gas under the assumptions that the gas particles are undergoing \emph{elastic}, \emph{localised} and \emph{microreversible} collisions and that the velocities of two particles at the same position are \emph{uncorrelated }before the collision, i.e. \emph{the molecular chaos assumption}. These collisions are described by the \emph{Boltzmann collision operator}, i.e. $\mc C$ in  \eqref{eq:kinetic} is given by  
\begin{align} \label{eq:Boltzmann} 
	\mc C [f] := Q (f, f) =  \int_{\R^d} \int_{\S^{d-1}} B (|v-v_*|, \cos \theta ) (f'f_*' - ff_*) \d \sigma \d v_*.  \footnotemark
\end{align} \footnotetext{where $f' := f(t,x,v')$, $ f_*' := f(t,x,v_*')$, $f = f(t,x,v)$ and $f_* := f(t,x,v_*)$ }The velocities of two particles before and after a collision occurs, $(v',v_*')$ and $(v,v_*)$ respectively, are given by
\begin{align} \label{eq:vel_law}
	v' = \frac{v+v_*}{2} +\frac{|v-v_*|}{2} \sigma \quad v_*' = \frac{v+v_*}{2} -\frac{|v-v_*|}{2} \sigma. 
\end{align} where $\sigma \in \S^{d-1}$ is a unit vector on the $(d-1)-$dimensional unit sphere. The angle between pre- and post- collisional velocities is called the \emph{deviation angle} and it satisfies that
\begin{align*}
	\cos \theta := \frac{(v_*' - v') \cdot (v_* - v)}{|v_*-v|^2} = \frac{v-v_*}{|v-v_*|} \cdot \sigma. 
\end{align*} In \eqref{eq:Boltzmann}, $B$ is called the \emph{Boltzmann collision kernel} and its shape depends on how the particles interact. We assume that $B$ can be written as
\begin{align} \label{eq:Boltzmannkernel}
	B(|v-v_*|, \cos \theta ) = |v-v_*|^\gamma b \left( \cos \theta \right), \qquad  \gamma = \frac{s-(2d-1)}{s-1},
\end{align} for $s>2$ (excluding the \emph{Coulomb interactions} case, i.e. $s=2$). The cases where $0 < \gamma \leq 1$,  $\gamma =0$ and $-d < \gamma <0$ correspond to \emph{hard potentials}, \emph{Maxwellian potentials} and \emph{soft potentials} respectively.

The assumption that the collision kernel is integrable with respect to the angular variable is known as the \emph{Grad's angular cut-off assumption}. Throughout the section, we consider that this assumption holds.

\paragraph{Well-posedness.} The study of  existence of solutions for the Boltzmann equation dates back to \cite{DL89} where DiPerna and Lions established a foundation of the global existence of solutions in the $L^1$ setting for the Grad's angular cutoff case. In \cite{AV02}, Alexandre and Villani extended this to the non-cutoff cases. Following these landmark results, it has become a subject for many other works. 
For a detailed discussion about the literature we refer to a recent survey \cite{S22} by Silvestre where he addresses recent regularity results and some open problems in this direction.

\paragraph{Long-time behaviour.}
Concerning the long-time behaviour of the Boltzmann equation, the famous \emph{H-theorem} states that the mathematical entropy $H(f)$ decreases in time, along the solutions of the Boltzmann equation, i.e.
\begin{align*}
	H[f] = \int_\O f \log (f) \d z, \quad \mbox{and} \quad \frac{\d }{\d t} H[f] = - \int_\O D[f] \d z \leq 0,
\end{align*} where $D[f] \geq 0$ is the so-called \emph{entropy dissipation}. 
Then there exists a unique global equilibrium state given by a \emph{local Maxwellian} (Gaussian distribution)
\begin{align*}
	f_\infty (x,v) = \frac{\rho}{(2 \pi T)^{d/2}} e^{- \frac{|v-u|^2}{2T}}:= M_{\rho, T, u} (x,v),
\end{align*} where the macroscopic quantities $\rho , u $ and $T$ are defined as \begin{align}
	\rho (t,x) &= \int_{\mc V} f(t,z) \d v, \tag{local density} \\
	u (t,x) &= \frac{1}{\rho (t,x)}\int_{\mc V} vf(t,z) \d v, \tag{local momentum} \\
	T (t,x) &= \frac{1}{d \rho(t,x)}\int_{\mc V} |v-u(t,x)|^2 f(t,z) \d v, \tag{local temperature}
\end{align} where $z=(x,v)$. The total mass, the total momentum and the total temperature of the system are conserved, i.e.
\begin{align*}
	\frac{\d}{\d t} \int_{\O \times \mc V} f (t,z) \begin{pmatrix}
		1\\
		v \\
		|v|^2/2
	\end{pmatrix} \d z = 0.
\end{align*}

In a prominent result \cite{DV05}, Desvillettes and Villani showed that the nonlinear Boltzmann equation converges to local Maxwellian at least at an algebraic rate provided that some a-priori strong Sobolev estimates are verified. Such estimates were established in spatially periodic domains near Maxwellians (see \cite{G95, G94}) but their validity in the spatially inhomogeneous setting in bounded domains remains still open. Furthermore, such estimates are not expected to hold in general non-convex domains (see e.g. \cite{G10} for further discussions).

There are previous works on the nonlinear Boltzmann equation, using a linearisation argument around the equilibrium (see e.g. \cite{U74,SA77}). The problem arising in this approach is that the convergence estimates are valid only once the system enters a small neighbourhood of the equilibrium.

\paragraph{Confinement.}A key mechanism to ensure the existence of an equilibrium state for spatially inhomogeneous kinetic equations is called the \emph{confinement}. There are several ways to induce confinement in a kinetic model (see e.g. \cite{V02} for this classification, other confinement mechanisms and a more detailed discussion on this matter): 
\begin{enumerate}[label=(\roman*)]
	\item \textbf{Torus confinement} simply refers to posing the equation on the torus, i.e. $x \in \T^d$. This is very convenient mathematically since the torus does not have a boundary. 
	\item \textbf{Potential confinement} means that the interaction between gas particles and their background is described via a \emph{confining potential} $\Phi(x)$, then the force is given by $F(x) = - \nabla_x \Phi (x)$. In this case, the potential $\Phi$ must satisfy $e^{-\Phi(x)} \in L^1(\R^d)$ to be confining.
	\item \textbf{Box confinement} refers to enclosing the system in a box or a vessel with boundary conditions. 
\end{enumerate}

Note that if e.g. $x \in \O \subset \R^d$, where $\O$ is a bounded set, even if a confining potential is considered, boundary conditions must be imposed. We will also consider cases where some boundary conditions are imposed in the absence of a collision term.

Throughout the section, we consider either $x \in \T^d$, or $x \in \R^d$. When $x \in \T^d$, the transport operator $\mc T$ in \eqref{eq:kinetic} is given by $\mc T = v \cdot \nabla_x f$. We define the normalised ($T=u=1$) global Maxwellian 
\begin{align} \label{Maxwell}
	M (v) : = \frac{1}{(2 \pi )^ {-d/2}}  e^{-\frac{|v|^2}{2}},
\end{align} which is the global equilibrium of the Boltzmann equation when $x \in \T^d$. Unless stated otherwise, when $x \in \R^d$ then we consider that the transport operator has also a macroscopic force term, i.e. $\mc T = v \cdot \nabla_x - \nabla_x \Phi (x) \cdot \nabla_v $, where $\Phi (x)$ is the confining potential.

\paragraph{Boundary conditions.}
%When the Boltzmann equation is posed on some bounded set but there is no external force acting on the particles, i.e. no confining potential, one way of establishing a confining mechanism is to complement the equation with some boundary conditions. 
The boundary conditions describe the interactions between the gas particles and the boundary of the domain which we denote as $\p \O$. Here we introduce some boundary conditions which are commonly used in the literature (we will also use them later). Moreover, we refer to e.g. \cite{CIP94, V02} for more detailed discussions on this topic.

Equation \eqref{eq:kinetic} may be supplemented with the boundary conditions given by
\begin{align} \label{general_boundary}
	\gamma_- f (t, z)= \mc R_{\gamma_+} f (t, z) \quad z \in \p \O \times \mc V,
\end{align}where $\gamma_\pm$ shows the trace of $f$ at the boundary sets $(0, \infty) \times \p_\pm \Sigma$ which are defined as 
\begin{align*}
	\p_\pm \Sigma := \{ z \in \p \O \times \mc V  \mid \pm ( v \cdot n_x) >0\},
\end{align*} where $n_x$ is the outward unit normal vector at $x \in  \p \O$. Here, $\Sigma_-$ and $\Sigma_+$ denote the incoming and the outgoing sets respectively. With this notation at hand, we define some common type of boundary conditions. 

\begin{enumerate}[label=(\roman*)]
	\item \textbf{Specular reflection} is the most natural boundary condition to consider and it means that the particles bounce back on the wall with the same pre- and post- collisional angles. The specular reflection is given by
	\begin{align}
		\label{specular_boundary}
		f(t, x,v)= f(t, x, R_x v), \quad R_xv = v - 2 (v \cdot n_x) n_x,
	\end{align} where $n_x$ is the outward unit normal vector at $x \in \p \O$. 
	\item \textbf{Diffuse reflection} means that the particles are adsorbed by the boundary before being re-emitted inside the domain according to the Maxwellian velocity distribution. The diffuse boundary condition, considering e.g. the incoming set, is given by
	\begin{align} 
		\label{diffuse_boundary}
		f(t, x,v) = \tilde R_+(x) M_{T_w}(x,v), \quad \tilde R_+(x) = \int_{\{v \cdot n_x >0\}} f(t, x,v) |v \cdot n_x| \d v,
	\end{align} where the Maxwellian $M_{T_w}(x,v) $ depending on the temperature of the wall (boundary) at $x \in \p \O$ denoted by $T_w(x)$ 
	\begin{align*}
		M_{T_w}(x,v) = \frac{c(x)}{(2 \pi T_w(x))^{d/2}} e^{\frac{-|v|^2}{2T_w(x)}}, \quad  c (x) = \left ( \int_{\{v \cdot n_x >0\}}   \frac{1}{(2 \pi T_w(x))^{d/2}} e^{\frac{-|v|^2}{2T_w(x)}} |v \cdot n_x| \d v \right )^{-1},
	\end{align*} for any $x \in \p \O$. It is defined similarly for $\tilde R_-$ where the integrals are taken over the set $\{v \cdot n_x<0\}$. %Moreover, both $\tilde R_-$ and $\tilde R _+$ may also be considered at the same time.
	%, e.g. to describe the boundary conditions on the right and on the left when the domain is an interval. 
\end{enumerate}
We notice that in the diffuse reflection case incoming and outgoing velocities aren't correlated for both the normal and tangential components. 
Moreover, it is physically irrelevant to assume that the particles will be reflected purely specular at the boundary. A more realistic description of the interactions at the boundary is given by the \emph{Maxwell boundary conditions} which combines both the specular and diffuse boundary conditions.
\begin{enumerate}[label=(\roman*)]
	\item[(iii)] \textbf{Maxwell boundary conditions} consist in taking a convex combination of the specular and the diffuse reflections, i.e. the reflection operator $\mc R_{\gamma_+}$ in \eqref{general_boundary} is given by
	\begin{align}
		\label{Maxwellgeneral}
		f(t,x,v) = \mc R_{\gamma_+} f (t,x,v): =  (1- \alpha (x)) f (x, R_x v) + \alpha (x) \tilde R_+ (x)M_{T_w} (x,v),
	\end{align} where $\alpha : \p \O \to [0,1]$ is a Lipschitz function called the \emph{accommodation coefficient}. We notice that $\alpha \equiv 0$ and $\alpha \equiv 1$ correspond to the pure specular reflection and the pure diffuse reflection respectively. 
\end{enumerate}
There are more complicated boundary conditions taking into account not only the effects of diffuse and specular reflections but also pressure of the particles on the boundary during an interaction. One of them is Cercignani-Lampis (CL) boundary conditions (see e.g. \cite{CIP94}) given below.

\begin{enumerate}[label=(\roman*)]
	\item[(iv)]  \textbf{Cercignani-Lampis boundary conditions} take into account that the probability distribution of the outgoing velocities keeps some information from the incoming velocities. In this case the reflection operatior $\mc R_{\gamma+}$ in \eqref{general_boundary} is given by
	\begin{align} \label{CL_boundary}
		\begin{split}
			\mc R_{\gamma+} f (t,x,v) =& \int_{\{u \cdot n_x >0\}} f(t,x,u) R (u \to v; x) |u \cdot n_x| \d u, \\
			R (u \to v; x) =& \frac{1}{T_w(x) r_{\perp}} \frac{1}{(2 \pi T_w(x)r_{\parallel} (2- r_{\parallel}))^{(d-1)/2}} e^{- \frac{|v_\perp|^2}{2T_w(x) r_\perp}} e^{-\frac{(1-r_\perp)|u_\perp|^2}{2T_w(x) r_\perp}} \\
			&\times e^{-\frac{|v_\parallel - (1-r_\parallel)u_\parallel|^2}{2T_w(x) r_\parallel (2-r_\parallel)}} I_0 \bigg( \frac{(1-r_\perp)^{1/2} u_\perp \cdot v_\perp}{T_w(x)r_\perp}\bigg),
		\end{split}
	\end{align} where 
	\begin{align*}
		v'_\perp := (v' \cdot n_x) n_x, \quad v'_\parallel := v'-v'_\perp, \quad  I_0 (y) := \frac{1}{\pi} \int_0^\pi e^{y \cos \theta} \d \theta,
	\end{align*} for any velocity $v'$, and $y \in \R$. The temperature of the boundary is denoted as $T_w(x)$ for all $x \in \p \O$. We denote the accommodation coefficients by $r_\perp \in (0,1]$ and $r_\parallel \in (0,2)$. %Note that $r_\perp = r_\parallel=1$ corresponds to the diffuse reflection.  
	We consider that the Cercignani-Lampis kernel $R$ is normalised (see the appendix of \cite{CH20} for a proof), i.e. for any $(x,u) \in \p_+ \Sigma$,
	\begin{align*}
		\int_{\{v \cdot n_x <0\}}R(u \to v;x) |v \cdot n_x| \d v =1.
	\end{align*} This normalisation ensures that the distribution of the outgoing velocities from $R$ is a probability distribution. Moreover, we note that, the case $(r_\perp, r_\parallel) = (1,1)$ corresponds to the diffuse reflection. 
\end{enumerate}
Finally, we remark that these boundary conditions conserve mass and positivity so that when we impose them we can still describe the process by a Markov semigroup (see Sections \ref{sec:Knudsen} and \ref{sec:nonlin_BGK}). We also mention a recent paper (cf. \cite{BCMT21}) where the authors studied the hypocoercivity for various linear kinetic equations with Maxwell boundary conditions. 

After this general introduction to the nonlinear Boltzmann equation, its properties and some boundary conditions, we look at the long-time behaviour of some kinetic models for interacting gas particles. These equations can be categorised into three (see \cite{V02} for this categorisation and some other variants):
\begin{enumerate}[label=(\roman*)]
	\item \textbf{Model equations:} the linear BGK, the nonlinear BGK equation on an interval, and the kinetic free-transport equation with boundary conditions describing the collisionless gas.
	\item \textbf{Linear models:} the linear Boltzmann and the linear degenerate Boltzmann equations.
	\item \textbf{Diffusive models:} the linear kinetic Fokker-Planck equation.
\end{enumerate}

We present the results following this order. 

\subsubsection{The linear BGK equation}

\label{sec:linearBGK}
The BGK (or relaxation Boltzmann) equation was introduced by the physicists Bhatnagar, Gross and Krook in 1954 (cf. \cite{BGK54}) as a simplified Boltzmann-like model for collision processes in gases. It satisfies the fundamental properties of the Boltzmann equation. Here we are interested in the linear version of it. In this case, the BGK collision operator is given by 
\begin{align} \label{eq:linBGK}
	\mc C [f] = \mc C_{ M}^+[f] - f =  M (v) \int_{\R^d} f ' \d v' - f,
\end{align} where $M$ is the normalised Maxwellian given by \eqref{Maxwell} and $f' := f(t,x,v')$.

In \cite{P89}, Perthame provided the first result on the global weak solutions to the linear BGK equation. The global well-posedness of the nonlinear BGK equation is also addressed in \cite{GP89}. In \cite{PP93}, the authors proved existence and uniqueness in bounded domains by using weighted $L^\infty$ bounds. Later in \cite{ZH07}, this result was extended to the $L^p$ setting. We also refer to a more recent paper \cite{WZ12} addressing the Cauchy problem for the BGK equation with an external force term.

Concerning the long-time behaviour, the exponential convergence to the unique equilibrium state, which is $f_\infty(x,v) = M(v) e^{-\Phi(x)}$, is shown in both the $L^2$ and $H^1$ settings (see e.g. \cite{DMS15,H06,CCG03,MN06}). Here we present the result of \cite{CCEY20} which shows quantitative geometric and subgeometric rates (depending on the confining potential) of convergence by using Doeblin and Harris theorems. This work also presents the first quantitative rate of convergence in the subgeometric setting.

\begin{theorem} \label{thm:linearBGK}
	Suppose that $t \mapsto f_t$ is the solution to the linear BGK equation \eqref{eq:kinetic}-\eqref{eq:linBGK} with initial data $f_0 \in \mc P (\O \times \mc V)$. We consider either of the three regimes below:
	\begin{itemize}
		\item[\textbf{R1.}] Suppose that $\O\times \mc V = \T^d \times \R^d$ (no confining potential, i.e. $\Phi =0$).
		\item[\textbf{R2.}] Suppose that $\O\times \mc V = \R^d \times \R^d$ and that the confining potential $\Phi \in C^2 (\R^d)$ satisfies that
		\begin{align}
			x \cdot \nabla_x \Phi (x) \geq \alpha |x|^2 + \beta \Phi (x) - \eta,
		\end{align} for some positive constants $\alpha, \beta, \eta$.
	\end{itemize}
	Then there exist positive constants $C, \lambda$ which are independent of $f_0$ such that
	\begin{align}
		\|f_t - f_\infty\|_{\phi} \leq C e^{-\lambda t} \|f_0 - f_\infty\|_{\phi},
	\end{align} where $f_\infty (x,v) = M(v)e^{-\Phi(x)}$ is the only equilibrium state and  $\phi =1 $ for \textbf{R1} (corresponds to the total variation or the $L^1$ distance) and for \textbf{R2}, the weight function is $\phi= 1+ \varphi$ where $\varphi$ is given by
	\begin{align} \label{BGKphi}
		\varphi (x,v) = H(x,v) +  \frac{x \cdot v}{4} + \frac{|x|^2}{8},
	\end{align} and $H(x,v) := \Phi(x) + |v|^2/2$ is the total energy of the system.
	\begin{itemize}
		\item[\textbf{R3.}] Suppose that $\O\times \mc V = \R^d \times \R^d$ and that the confining potential $\Phi \in C^2 (\R^d)$ satisfies, for some $\xi \in (0,1)$, that
		\begin{align}
			x \cdot \nabla_x \Phi (x) \geq \alpha \langle x \rangle ^{2 \xi} + \beta \Phi (x) - \eta,
		\end{align} for some positive constants $\alpha, \beta, \eta$.
	\end{itemize}
	Then there exists a positive constant $C$ such that 
	\begin{align*}
		\|f_t - f_\infty\|_{1}  \leq \min \left \{ \|f_0 - f_\infty\|_{1}, C (1+t)^{-\frac{\xi }{ 1-\xi } } \|f_0\|_\phi \right \},
	\end{align*} where $f_\infty (x,v)= M(v)e^{-\Phi(x)}$ is the only equilibrium state and the weight function is $\phi = \varphi^\xi$, where $\varphi$ is given by \eqref{BGKphi}.
\end{theorem}

\begin{proof}[Idea of the proof of Theorem \ref{thm:linearBGK}] Here we give the ideas of the proofs in three different regimes. For detailed computations we refer to Section 3 of \cite{CCEY20}. 
	
	We start by writing the solutions of \eqref{eq:kinetic} with \eqref{eq:linBGK} as  
	\begin{align} \label{firstDuhamel}
		f_t \geq e^{-t}\int_0^t \int_0^s \mc T_{t-s} \mc C_M^+ \mc T_{s-r} \mc C_M^+ \mc T_r f_0 \d r \d s.
	\end{align} This is done by repeatedly using Duhamel's formula. The rest of the proof relies on showing a lower bound on the actions of the operators $\mc C_M^+$ and $\mc T$ individually. Note that $(\mc T_t)_{t\geq 0}$ is the transport semigroup describing the action of $\mc T$ at each time $t \geq 0$. This part is different for each regime and the difference arises in the confining potential, more precisely the assumptions made on the confining potential. It is expected to have a slower convergence when the confinement is not strong enough. In the presence of a confining potential, the external force exerted on particles will deviate their trajectories from a straight line  (which is the result of the free transport), we need to track these trajectories to estimate the action of $\mc T$. As our starting point, we always take an initial data $f_0(z) = \delta_{(x_0,v_0)}(x,v)$ where $(x_0, v_0) \in \O \times \mc V$ is an arbitrary point. Since the underlying semigroup is Markov, later we can extend this to any probability measure.
	
	\paragraph{\textit{R1.}} In this case, the proof uses Doeblin's theorem (Theorem \ref{thm:Doeblin}). Thus we show how to verify \ref{Doeblin}. The idea is to show that $\mc C_M^+$ allows jumps to any small velocity $v$ and that the transport operator $\mc T$ allows reaching anywhere in a ball of radius $R$ once the starting point is any point inside $R$ with any starting velocity $v$.
	
	More precisely, for any $ \beta>0$ we can find an $\varepsilon >0$ small such that for any probability measure $\nu \in \mc P (\T^d \times \R^d)$ 
	\begin{align*}
		\mc C_M^+ \nu (x,v) \geq \varepsilon \mathds{1}_{\{|v| \leq \beta\}} \int_{\R^d} \nu (x,v') \d v'. 
	\end{align*} Notice that this is easily verified by choosing $\varepsilon = M(v)$ for any $v$ such that $|v|= \beta$. 
	
	For the transport part, given any time $\tau>0$, for all $t \geq \tau$, for any $x_0 \in \T^d$ we can find $\beta, \tilde R>0$ s.t. the following bound holds
	\begin{align*}
		\int_{B(0, \tilde R)\footnotemark} \mc T_t \left( \delta_{x_0} (x) \mathds{1}_{B(0, \beta)} (v) \right) \d v \geq \tau^{-d}.
	\end{align*} %for $\beta > 2R/t$ and $\tilde R >2R/t$. 
	By combining the last two estimates, we show that there exists $\tau>0$ such that 
	\begin{align*}
		S_{\tau} f_0 := f_{\tau} \geq \frac{1}{9}e^{-\tau} \tau^{2-d} \varepsilon^2 \mathds{1}_{\{|v| \leq \beta\}}.
	\end{align*} Thus, \ref{Doeblin} is satisfied and we can find $C$ and $\lambda$ explicitly. 
	\footnotetext{$B(z, R)$ denotes the open ball centered at $z$ with a radius $R$.}
	\paragraph{\textit{R2.}} In this case, it is easy to check the weight function \eqref{BGKphi} verifies the \ref{Foster_Lyap}, i.e. \eqref{eq:FL3} with $\zeta = \min \{\alpha, \beta, 1\}/4$ and $D = d/2 + \eta/4$. 
	
	For the \ref{Minorisation}, the estimate coming from the\textbf{ \em R1} case is not enough as we need to track the trajectories of particles under the action of the potential $\Phi$. The characteristics associated to the transport part in this case are given by
	\begin{align} \label{eq:ODE1}
		\begin{split}
			\dot x &= v, \\
			\dot v &= - \nabla \Phi(x).
		\end{split}
	\end{align} Denoting the solution of \eqref{eq:ODE1} with the initial data $(x(0)=x_0, v(0)=v_0) \in \O \times \mc V$ by $(X_t(x_0, v_0), V_t (x_0,v_0))$ we have 
	\begin{align*}
		X_t (x_0,v_0) = x_0 + v_0 t + \int_{0}^{t} \int_{0}^{s} \nabla \Phi (X_r(x_0,v_0)) \d r \d s,
	\end{align*} for any $(x_0,v_0) \in \R^d \times \R^d$ and for any $t$ where it is defined. The idea is to approximate $(X_t, V_t)$ by $(X_t^0, V_t^0)$, by choosing a pair $(X_t^\varepsilon , V_t^\varepsilon)$ between $(X_t, V_t)$ and $(X_t^0, V_t^0)$, which solves the following ODE:
	\begin{align} \label{eq:ODE2}
		\begin{split}
			\dot x &= v, \\
			\dot v &= - \varepsilon^2\nabla \Phi(x),
		\end{split}
	\end{align} with the initial data $(x_0,v_0) \in \R^d \times \R^d$. Note that  $(X_t^0, V_t^0)$ is the solution of \eqref{eq:ODE2} when $\varepsilon = 0$. We see that 
	\begin{align*}
		X_t^\varepsilon (x_0, v_0) = X_{\varepsilon t} \left(x_0, v_0/ \varepsilon \right) \quad  \mbox{and} \quad  V_t^\varepsilon (x_0, v_0) = \varepsilon V_{\varepsilon t} \left(x_0, v_0/ \varepsilon \right). 
	\end{align*} For a fixed $f$, and $\varepsilon \in [0,1]$ we define a $C^1$ map $F(\varepsilon,v) = X_t^\varepsilon (x_0,v)$. Then for any $\varepsilon =0$ we can find $v_*$ such that $F(0, v_*) = x_1$. Then, since $\nabla F(0, v_*) \neq 0$ by the implicit function theorem, for all $\varepsilon < \varepsilon_*$ there is a $C^1$ function $v(\varepsilon)$ such that $F(\varepsilon, v (\varepsilon)) =x_1$, i.e. $X_{\varepsilon t} (x_0, v(\varepsilon)/\ve) = x_1$. If $s$ is taken such that $s < \ve_*t$, then we can choose $v$ such that $X_s(x_0, v) = x_1$. We can obtain quantitative estimates on $\ve_*$ by finding bounds on $(X_t,V_t)$ and $\nabla \Phi (X_t)$ for $t>0$ in fixed intervals. The details for this can be found in Lemmas 3.4 and 3.5 of \cite{CCEY20}. Finally, for a given $R>0$ there exists a time $\tau >0$ and for any $0< s< \tau$ we can find $\alpha, \beta, \tilde R >0$ explicitly such that \begin{align} \label{ineq3}
		\int_{B(0, \tilde  R)} T_s (\delta_{x_0} (x) \mathds{1}_{\{ |v| \leq \beta\}}) 	\d v \geq \alpha \mathds{1}_{\{|x| \leq R\}},
	\end{align} for any $|x_0| \leq R$. We extend this argument to any initial data which is compactly supported in $x$ and $v$ on some ball and this verifies the \ref{Minorisation}. Then, the result follows from Theorem \ref{thm:harris}.
	\paragraph{\textit{R3.}} The \ref{Minorisation2} is verified in the same way as the \textbf{\em R2} case above.
	
	For the \ref{Foster_Lyap2}, we use the weight function \eqref{BGKphi} and it is easy to check that $\phi^\xi$ verifies the inequality of the \ref{Foster_Lyap2} with $\zeta = \min\{\alpha, \beta,1 \}/4$, for some $D>0$. We can take $V(s) = 1 + s^{\xi}$, then we have for large $u$
	\begin{align*}
		G_V(u) = \int_1^u \frac{\d s}{V(s)} =  \int_1^u  \frac{\d s}{1+s^{\xi}}  \sim 1 + u^{1-\xi},
	\end{align*} and for $t$ large we have
	\begin{align*}
		G_V^{-1} (t) \sim 1 + t^{\frac{1}{1-\xi}}, \quad \mbox{and} \quad V \circ 	G_V^{-1} (t)  \sim 1 + t^{{\frac{\xi}{1-\xi}}}.
	\end{align*} Therefore the result follows from Theorem \ref{thm:subgeoHarris}.
\end{proof} 

In the next two sections, we will look at how the boundary conditions affect the long-time behaviour of kinetic equations. We continue with the free-transport equation which describes the collisionless (Knudsen) gas.

\subsubsection{The kinetic free-transport equation with boundary conditions} 
\label{sec:Knudsen}

In this section, we study the kinetic free-transport equation, i.e. $\p_t f + v \cdot \nabla_x f = 0$. In this case, the average distance travelled by the gas particles between collisions is greater than the size of the domain so that the interaction between gas particles is negligible, i.e. $\mc C = 0$ in \eqref{eq:kinetic}. A gas in such a low-density regime is called a \emph{Knudsen gas}. The property of a Knudsen gas is that the \emph{Knudsen number}, which is the ratio between the mean free path and the length scale related to the size of the domain, is high. 

The particles follow the free-transport at a constant speed, i.e. $\mc T  = v \cdot \nabla_x $, until they hit the boundary. After hitting the boundary they get diffused and reflected according to some boundary conditions, i.e. \eqref{general_boundary}. This equation is used as a model for the evolution of a gas enclosed in a vessel $\O \subseteq \R^d$.  We refer to \cite{AC93} for the well-posedness of this equation in the $L^1$ setting. 

The lack of spectral gap for the kinetic free-transport equation (since there is no dissipative (collisional) part) makes the study of quantitative convergence rate a difficult problem and the expected rate is slower than exponential (see e.g. \cite{TAG10} for a numerical study). The relaxation to equilibrium depends on several factors such as the boundary conditions, and the shape and space dimension of the vessel. In \cite{AG11}, the authors proved rigorously the algebraic convergence rate considering the kinetic free-transport equation with diffuse reflection at the boundary which is kept at a uniform, constant temperature. They also proved an exponential convergence rate when the density of molecules is monokinetic. This essentially excludes the ``negative'' effect of the \emph{slow} particles on the equilibration. In \cite{KLT13}, the authors considered the diffuse reflection at the boundary with a constant temperature and a bounded symmetric domain in $\R^d$, $d = \{1,2,3\}$. They showed that the solution converges to the global Maxwellian with a given temperature with the optimal convergence rate $(1+t)^d$. Later in \cite{KLT14}, the same authors improved this result by considering variable boundary temperature and obtained the same optimal decay rate. Then in \cite{K15}, the author considered the equation in a spherical domain in $\R^d$, $d = \{1,2\}$ with the Maxwell boundary conditions and obtained an algebraic convergence rate. In \cite{BF22}, the authors could remove the symmetry assumption on the domain. More precisely, they obtained a polynomial convergence rate for the free-transport equation in a $C^2$ bounded domain in $\R^d$, $d=\{2,3\}$ with Maxwell boundary conditions. The Maxwell boundary conditions have a lesser thermalising effect than the diffuse reflection and this leads to slower convergence rates. More precisely the constant before the convergence rate, e.g., $C$ in \eqref{eq:hypocoercivity1}, is bigger, but the exponent of the polynomial convergence is unchanged. This comes from the fact that the specular reflection does not have any thermalising effect and does not contribute towards the convergence to equilibrium. Moreover, most of the aforementioned results are obtained by using several probabilistic arguments. We refer also e.g. to \cite{LMM_K21} where the authors study the convergence equilibrium for the free-transport equation with diffuse boundary conditions using the Tauberian approach.

In this section, we consider either the Maxwell boundary conditions \eqref{Maxwellgeneral} or the Cercignani-Lampis boundary conditions \eqref{CL_boundary} and state the main results of \cite{B20, B21}. In \cite{B20}, the author uses the subgeometric version of Harris's theorem (Theorem \ref{thm:subgeoHarris}) to obtain the polynomial rate of convergence for the kinetic free-transport equation describing the evolution of a Knudsen gas in a vessel in $\R^d$, $d = \{2,3\}$ with the Maxwell boundary conditions. This result does not require a symmetry assumption on the domain and it is extended to the case where the temperature varies at the boundary. In \cite{B21}, the same author gave the first quantitative subgeometric convergence results on this equation, with the Cercignani-Lampis boundary conditions and the result in \cite{B21} is optimal.

First, we define a map $\tilde \tau$ which gives the time of the first collision with the boundary for a particle starting from $z =(x,v) \in \O \times \mc V$ at time $0$, 
\begin{align*}
	\tilde \tau (x,v) = \begin{cases}
		\inf \{ t >0, \, x + tv \in \p \O\}, \quad &(x,v) \in \p_- \Sigma \cup \Sigma, \\
		0, \quad & (x,v) \in \p_+ \Sigma \cup \p_0 \Sigma.
	\end{cases}
\end{align*} where $\p_0 \Sigma : = \{z \in \p \O \times \mc V \mid v \cdot n_x= 0\}$. %We remark that $\tilde \tau$ is the time of the first collision when 

The main results of \cite{B20, B21} are given by

\begin{theorem} \label{thm1_collisionless gas}
	Suppose that $t \mapsto f_t$ is the solution to the free-transport equation $\p_t f + v \cdot \nabla_x f =0$ with initial data $f_0 \in \mc P_{\phi_d}(\O \times \mc V)$ where $\O \subseteq \R^d,$ with $d \in \{2,3\}$, $\O$ is a $C^2$ bounded domain and $ \mc V =\R^d$. The equation is complemented with boundary conditions \eqref{general_boundary}. We assume that $T_w(x) : \p \O \to \R_+$ is a continuous function, positive on $\p \O$ compact. We consider either of the following regimes below:
	
	\begin{itemize}
		\item [\textbf{R1.}] Suppose that the kinetic free-transport equation is complemented with the Maxwell boundary conditions \eqref{Maxwellgeneral} and given any constants $\alpha_0 \in (0,1)$ and $\alpha_1 \in (0,1)$  such that 
		\begin{alignat*}{3}
			\alpha (x) &\geq \alpha_0, \quad &x &\in \p \O, \\
			(1-\alpha_1)^i &\geq (1- \alpha_0) \quad &\mbox{for }&  i \in \{1,2,3,4\},
		\end{alignat*} where $\alpha (x)$ is the accommodation coefficient in \eqref{Maxwellgeneral}.
		\item [\textbf{R2.}] Suppose that the kinetic free-transport equation is complemented with the Cercignani-Lampis boundary conditions \eqref{CL_boundary} with  $r_\perp \in (0,1)$ and $r_\parallel \in (0,2)$, and that there exist $T_{\min}, T_{\max} > 0$ such that for all $x \in \p \O$
		\begin{align*}
			0 < T_{\min} \leq T_w(x) \leq T_{\max},
		\end{align*} where $T_w$ is the temperature of the boundary. 
	\end{itemize}
	Then there exists a unique equilibrium state $f_\infty \in \mc P_{\phi_d}(\O\times \mc V)$ for \textbf{R1} and $f_\infty \in \mc P_{\phi_{d,\ve}}(\O\times \mc V)$ for \textbf{R2} and there exist a positive constant $C$ and a function $\beta (t,d)$ for $d \in \{2,3\}$ such that 
	\begin{align} \label{eq:subgeo1}
		\|f_t - f_\infty\|_1 \leq C \beta (t,d)\|f_0 - f_\infty\|_{\phi_d},
	\end{align} where 
	\begin{align}
		\beta (t,d) &= \frac{\log (1+t)^{d+1}}{(1+t)^{d+1}} \tag{for \textbf{R1}},\\
		\beta(t,d) &=  \frac{1}{(1+t)^{d- \ve}}, \tag{for \textbf{R2}}
	\end{align}
	and the weight function $\phi_d$ is given by 
	\begin{align}
		\phi_d &= \left( e^2 + \frac{\diam (\O)  }{\alpha_1|v|} - \tilde \tau (x, -v)\right)^d  \log \left( e^2 + \frac{\diam (\O)}{\alpha_1 |v |} - \tilde \tau (x, -v)\right)^{-1.6 \frac{d}{d+1}} \tag{for \textbf{R1}}, \\
		\phi_d &= 	\phi_{d, \ve} =
		(1 + \tilde \tau (x,v) + \sqrt{|v|})^{d-\varepsilon}, \quad \varepsilon \in (0,1/2 ) \quad \mbox{and} \quad \phi_d = \phi_{d,0} \quad \mbox{for all  }  (x,v) \in \bar{\O} \times \bar{ \mc V}  \tag{for \textbf{R2}}, 
	\end{align} where $\diam (\O)$ is the diameter of $\O$ and $\bar{\O} \times \bar{\mc V} $ is the closure of the set $\O \times \mc V$. Moreover the constant $C$ is a different one for each regime.
\end{theorem}
\begin{proof} [Idea of the proof of Theorem \ref{thm1_collisionless gas}] 
	The proof uses Theorem \ref{thm:subgeoHarris} for both the regimes. We give the sketch of the proofs separately for the two regimes. For the details we refer to \cite{B20,B21} for \textbf{\em R1} and \textbf{\em R2} respectively.
	\paragraph{\textit{R1.}} 	First we define a family of weight functions
	\begin{align*}
		\phi_i= \left( e^2 + \frac{\diam (\O)  }{\alpha_1|v|} - \tilde \tau (x, -v)\right)^i \log \left( e^2 + \frac{\diam (\O)}{\alpha_1|v|} - \tilde \tau (x, -v)\right)^{-1.6 \frac{d}{d+1}},
	\end{align*} for $i \in \llbracket d-1, d+1 \rrbracket$ where $ d \in \{2,3\}$ and define also $\phi_0$ such that $1 \leq \phi_0 \leq \phi_1 \leq \phi_{d+1}$. 
	
	To verify the \ref{Minorisation2} we show that there exists $\tilde R>0$ such that for all $R > \tilde R$, there exist some time $\tau(R)>0$ (depending on $R$) and a nonnegative measure $\eta (z)\in \mc M (\O\times \mc V)$ such that for all $z \in \O \times \mc V$
	\begin{align} \label{Doeblin_B1}
		S_\tau f_0 (z):= f_\tau (z)\geq \eta (z) \int_{\{\tilde \tau (z') \leq R\}} f_0 (z')   \d z' .
	\end{align} This bound is proved by tracking the characteristics of the transport equation similar to the previous sections.
	
	For the weak Foster-Lyapunov condition, we first prove that for all $\tau>0$ and $f_0 \in \mc P_{\phi_{d+1}} (\O\times \mc V)$
	\begin{align} \label{Lyap_subgeo}
		\|f_\tau\|_{\phi_{d+1}} + C \int_0^\tau \|f_s\|_{\phi_{d}}  \d s \leq \|f_0\|_{\phi_{d+1}} + \tilde C (1+\tau) \|f_0\|_1,
	\end{align} by using the fact that $v \cdot \nabla_x \tilde \tau (z) = -1$ for all $z \in \O \times \mc V$. Then by using \eqref{Doeblin_B1} we obtain
	\begin{align*}
		|||	S_\tau f_0 |||_{\phi_{d+1}} \leq |||f _0 |||_{\phi_{d+1}},
	\end{align*} where for a fixed $\tau >0$ large enough $ |||\cdot |||_{\phi_{d+1}} := \|\cdot \|_1 + \gamma_1 (\tau )\|\cdot \|_{\phi_{d+1}}  + \gamma_2 (\tau)\|\cdot\|_{\phi_{d}} $ for some $\gamma_1, \gamma_2 >0$ depending on $\tau$. 
	Using this repeatedly together with the inequality $|||S_\tau f_0|||_{\phi_1}  + \gamma_3 \|f_0\|_{\phi_0}  \leq |||f_0|||_{\phi_1}$ for a well-chosen $\gamma_3 >0$ and subtracting $f_\infty$ from both sides we obtain \eqref{eq:subgeo1}. We would like to remark that, the version of the \ref{Foster_Lyap2} used in this section is rather an integrated version of the condition of Theorem \ref{thm:subgeoHarris}.

	\paragraph{\textit{R2.}} The proof is similar to the case of \textbf{\em R1} except for two key points. We remark here only the differences and we refer to \cite{B21} for the details. 
	
	Similar to \eqref{Lyap_subgeo} we would like to obtain
	\begin{align*} 
		\|f_\tau\|_{\phi_{d+1-\ve}} + C \int_0^\tau \|f_s\|_{\phi_{d-\ve}}  \d s \leq \|f_0\|_{\phi_{d+1-\ve}} + \tilde C (1+\tau) \|f_0\|_1,
	\end{align*} for $\ve \in (0,1/2)$. 
	In the first regime \textbf{\em R1}, for the diffuse part of the Maxwell boundary conditions, the outgoing velocities are independent from the incoming velocities, thus the flux can be controlled uniformly. For the Cercignani-Lampis boundary conditions, outgoing velocities depend on the incoming ones, this results in an integration only on some part of the boundary. The second point is that due to the hypotheses made on $r_\perp$ and $r_\parallel$, the norm of the velocity decreases on average at each collision with the boundary. These lead to more complicated computations for the \ref{Minorisation2}.
	
	The \ref{Foster_Lyap2} is verified in a similar strategy to the regime \textbf{\em R1}.
	
	We conclude by noting that in the regime \textbf{\em R2}, the rate of the convergence is optimal.
\end{proof}

\begin{remark}
	In \cite{B20}, the author compares the result of \textbf{ \em R1} with the same initial value problem complemented with the absorbing boundary conditions, i.e. $\mc R_{\gamma_+} = 0$ in \eqref{general_boundary}. Notice that in this case the mass is not conserved and the equilibrium distribution is $f_\infty (z)=0$. The following convergence rates verify the Inequality \eqref{eq:subgeo1}:
	\begin{align*}
		\beta(t,d) &= \beta (t) = e^{-t}  \mbox{ when } \phi_d = e^{\tilde \tau (z)}, \\
		\beta(t,d) &= \beta (t) = (1+t)^{-\xi} , \mbox{ when } \phi_d =  (1 + \tilde \tau (z))^\xi,\, \xi >1 .
	\end{align*}
	Moreover, for $f_0 \in \mc P (\O\times \mc V)$ 
	satisfying $f_0 \1_{\{|v| \leq \delta \}} = 0$ for some $\delta >0$, the convergence rate is exponential in the case with absorbing boundary conditions, whereas the rate is of order $d+1$ in the case with Maxwell boundary conditions.
	
\end{remark}

\subsubsection{The nonlinear BGK equation on an interval}
\label{sec:nonlin_BGK}

In this section, we will look at a nonlinear BGK equation on an interval in dimension $d=1$ with diffuse boundary conditions, i.e. we couple the system with two heat reservoirs at different temperatures. 
We consider \eqref{eq:kinetic} in the phase space $\O \times \mc V = (0,1) \times \R$ with
\begin{align} \label{eq:nonlinBGK}
	\mc T [f] : = v \cdot \p_x f , \quad \mbox{and} \quad  \mc C[f] : = \frac{1}{\kappa} (\rho_f M_{T_f}- f), 
\end{align} where the density $\rho_f$ and the pressure $P_f$ are given by
\begin{align*}
	\rho_f (x) := \int_\R f \d v \quad \text{and} \quad P_f(x):=\rho_f(x) T_f(x) := \int_\R v^2 f \d v,
\end{align*} and $M_{T_f}$ is the Maxwellian velocity distribution with the temperature $T_f$, $\kappa>0$ is the Knudsen number. Equation \eqref{eq:kinetic}-\eqref{eq:nonlinBGK} is complemented with the diffuse boundary conditions, for $i \in \{0,1\}$
\begin{align} \label{eq"ninlinBGK2}
	f(t, i, v) =  \tilde R_i M_i(v), \quad 	M_i =  \frac{1}{T_i } e^{-\frac{v^2}{2T_i}}
\end{align} where $M_i$ are the normalised Maxwellian velocity distributions at the boundaries, i.e. $\int_0^\infty v M_i(v) \d v =1$. Note that at $x= 1$, we consider $\tilde R_1 := \tilde R_+$ as defined in \eqref{diffuse_boundary} and at $x=0$, we take $\tilde R_0 := \tilde R_+$ which is defined similarly to \eqref{diffuse_boundary} where the integral is taken on the set $\{ v <0\}$ instead of $\{ v >0\}$, i.e. $n_x=1$ in \eqref{diffuse_boundary}. 

In \cite{EM21}, the authors studied this model and they showed the existence of non-equilibrium steady states induced by the boundary conditions by using Doeblin's theorem. Some previous works on this equation can be found in \cite{CELMM18,CELMM20} and references therein. 

We also mention a numerical study \cite{KKSRD11} where authors studied the stationary states of a rarefied gas between parallel plates with sinusoidal temperature distribution. They considered two types of boundary conditions including the Cercignani-Lampis boundary conditions and they provided a numerical study on the non-equilibrium steady states. Particularly, their experiments show that the steady states differ drastically depending on the accommodation coefficients in the boundary conditions and the temperature distribution.

The main result of \cite{EM21} is given by

\begin{theorem} \label{thm:nonlin_BGK}
	Suppose that $t \mapsto f_t$ is the solution of the nonlinear BGK equation \eqref{eq:kinetic}-\eqref{eq:nonlinBGK} in the phase space $\O\times \mc V = (0,1) \times \mc \R$ with the initial data  $f_0 \in \mc P ((0,1) \times \mc \R)$.  Then for every fixed two temperatures $T_1, T_2$ satisfying
	\begin{itemize}
		\item[\textbf{A1.}]  $T_1 > \alpha_1/\kappa^2$, 
		\item[\textbf{A2.}]  $\sqrt{T_2} - \sqrt{T_1} \geq \alpha_2 \sqrt{\kappa} T_2^{1/4}$
	\end{itemize} where $\alpha_1, \alpha_2 >0$ are two positive constants and $\kappa>0$ is the Knudsen number given in the collision operator \eqref{eq:nonlinBGK}, there exists a nontrivial stationary state $f_\infty \geq 0$ of \eqref{eq:kinetic}-\eqref{eq:nonlinBGK} and it has the following properties
	\begin{itemize}
		\item [\textbf{P1.}] The sationary state has a zero momentum uniformly in $x\in (0,1)$, i.e.  
		\begin{align*}
			u_\infty(x) = \frac{1}{\rho_f(x)}\int_\R v f_\infty \d v =0.
		\end{align*}
		\item [\textbf{P2.}] The density of the stationary state is constant, i.e. $\rho_f (x)= \rho_f$, and it satisfies for all $x \in (0,1)$, that
		\begin{align*}
			1- C_1\sqrt{\kappa} T_1^{-1/4} \leq \rho_f(x) \leq 1+ C_2 \sqrt{\kappa} T_1^{-1/4}.
		\end{align*} 
	\item [\textbf{P3.}] The staionary state has a constant pressure, i.e. $ P_f(x)= P_f$, and it satisfies for all $x \in (0,1)$, that
	\begin{align*}
	C_3 \sqrt{T_1T_2} \leq P_f(x) \leq C_4 \sqrt{T_1T_2}.
	\end{align*}
		\item [\textbf{P4.}] The temperature of the steady state $T_f$ satisfies for all $x \in (0,1)$, that
		\begin{align*}
			C_5 \sqrt{T_1T_2} (1- C_2 \sqrt{\kappa} T_1^{-1/4}) \leq T_f(x) \leq C_6\sqrt{T_1T_2} (1+  C_1 \sqrt{\kappa} T_1^{-1/4}),
		\end{align*} 
	\end{itemize} where $C_1, C_2, C_3, C_4, C_5, C_6$ are positive constants. 
\end{theorem}

\begin{proof}[Idea of the proof of Theorem \ref{thm:nonlin_BGK}]
	We first look at the linear equation $\p_t f + v \p_x f = \rho(x)M_{T(x)} -f$ with the same diffuse boundary conditions for a given temperature profile $T(x)$ and show that this equation has a unique equilibrium state. This part uses Doeblin's theorem as the existence of a unique equilibrium is a by-product of the theorem.
	
	By Duhamel's formula we have for $x-vt \in (0,1)$ that
	\begin{align} \label{Duhamel1_nonlin}
		f_t (x,v)= e^{-t}f_0 (x-vt, v) + \int_0^t e^{-(t-s)} \rho (x-v(t-s))M_{T(x-v(t-s))} (v) \d s,
	\end{align} and similarly we obtain two more expressions at the boundaries considering $x-vt \leq 0$ at $x=0$ and $x-vt \geq 1$ at $x=1$. Combining these and defining $\mc C^+ [f] := \rho_f(x)M_{T(x)}(v)$ we obtain
	\begin{align} \label{eq:rev1}
		f_t (x,v)\geq \int_0^F e^{-F(t,x,v)} (\mc C^+ f) (s, x-v(F(t,x,v)-s),v) \d s,
	\end{align} where $F$ is defined as
	\begin{align*}
		F(t,x,v) = \begin{cases}
			x/v, \quad &\mbox{for  } x/v \leq t, \\
			t, \quad &\mbox{for  } x/v >0,\, v >0 \\
			t, \quad &\mbox{for  } v=0\\
			t, \quad &\mbox{for  } (1-x)/|v| >t,\, v <0 \\
			(1-x)/|v|\quad &\mbox{for  } -(1-x)/v \leq t.\\
		\end{cases}
	\end{align*} 
	Then we estimate $f_t$ at the boundaries by using the first expression on the right hand side of  \eqref{Duhamel1_nonlin} and using this we look at the local density $\rho$ for $v_0<0$ and $v_0>0$. We have that $\rho (t,x) = \delta_{(x_0 + v_0t)} (x)$ in the case where $x_0 + v_0t \in (0,1)$. We consider $x-vt \in (0,1)$ and since
	\begin{align*}
		M_{T(x)}(v) \geq \frac{1}{\sqrt{2\pi T_2}} e^{-\frac{v^2}{2T_1}} \geq \sqrt{\frac{T_1}{T_2}}M_{T_1} (v):= \gamma G(v),
	\end{align*} we can bound \eqref{Duhamel1_nonlin} as
	\begin{align} \label{Duhamel2_nonlin}
		f_t (x,v)\geq \gamma G(v) \int_0^t e^{-(t-s)} \rho (x-v(t-s)) \d s.
	\end{align} Now we consider \eqref{Duhamel2_nonlin} in three different cases for an initial data $f_0(x,v) = \delta_{(x_0,v_0)}(x,v)$,
	\begin{align*}
		\begin{cases}
			v_0 <0, \quad x_0 \leq \ve |v_0|, \\
			v_0 >0, \quad(1-x_0)\leq \ve v_0 , \\
			\mbox{otherwise}.
		\end{cases}
	\end{align*} We first estimate $\rho(t,x)$ in these three cases separately and then substitute these estimates in \eqref{eq:rev1}. Each regime involves some computations which can be found in the proof of Lemma 13 in \cite{EM21}. Then we obtain for a fixed $\ve>0$, setting $\tau= 2\ve$ that 
	\begin{align*}
		S_\tau f_0 := f_\tau \geq \beta \1_{\{x-2v\ve \in (\ve,1-\ve)\}},
	\end{align*} where 
	\begin{align*}
		\beta = \gamma e^{-2 \ve} \min \left\{ \frac{\gamma\ve}{2} G \left  (\frac{2}{\ve} \right ), \frac{1}{T_1} e^{-\frac{1}{2T_1\ve^2}}, \frac{1}{T_2}e^{-\frac{1}{2T_2\ve^2}}\right\}.
	\end{align*} Thus we verified \ref{Doeblin} and by Doeblin's theorem we obtain a unique stationary state for the linear equation. 
	
	The next step is to prove that this equilibrium state is also a stationary solution of the nonlinear equation. This is done by using Schauder's fixed point theorem. We define a map $\mc F : C((0,1)) \to C ((0,1))$ given by 
	\begin{align*}
		\mc{ F }(T)(x) = \frac{\int v^2 f \d v}{\int f \d v} = \tilde T_T (x),
	\end{align*}
	which is a map between continuous functions on $(0,1)$ and $\tilde T_T$ is the temperature profile of the solution of the linear BGK equation with initial temperature profile $T$. Using the fact that, 
	\begin{align*}
		\mbox{``If $T_1 \leq T(x) \leq T_2$ for $T_1,T_2$ satisfying \textbf{\em A2}, then $T_1 \leq \tilde T_T (x) \leq T_2$ and $\tilde T_T$ is $1/2-$H\"{o}lder continuous''}
	\end{align*} we show that $\mc F$ has a fixed point. This means that there exists a temperature profile $T_0$ such that $\tilde T_{T_0} (x) = T_0 (x)$ for all $x$, thus the equilibrium states, of the linear and the nonlinear models will be the same. For the details we refer to \cite{EM21}.
\end{proof}

\subsubsection{The linear Boltzmann equation}

\label{sec:lin_Boltzmann}

In this section, we are interested in the linear Boltzmann equation where the collision operator $\mc C$ in \eqref{eq:kinetic} is given by
\begin{align} \label{eq:linearBoltzmann}
	\mc C [f] = Q (f, M) :=  \int_{\R^d} \int_{\S^{d-1}} B (|v-v_*|, \sigma) (f'M(v_*') - fM(v_*)) \d \sigma \d v_*, 
\end{align} and $M$ is the Maxwellian velocity distribution given by \eqref{Maxwell}, $B$ is the Boltzmann collision kernel given by \eqref{eq:Boltzmannkernel} with $\gamma \geq 0$ and $\sigma \in \S^{d-1}$.
The linear Boltzmann equation describes the interaction between gas particles scattering with the background medium which is considered to be already in equilibrium. There is no self-interaction between the particles. It is used as a model for many systems like radiative transfer, cometary flow, dust particles and neutron transport, see e.g.  Chapter 1 of \cite{V02} and references therein. 
There are several ways of parametrising the pre-and post-collisional velocities and depending on which parametrisation we consider we can define the collision operator $\mc C$ differently. We will give some of them below since different representations will be used in different parts of the proofs later. More details on these representations can be found in e.g. \cite{V02}.

\begin{enumerate}[label=(\roman*)]
	\item \textbf{$\sigma-$representation.} Under the Grad's angular cut-off assumption, we can write the collision operator $\mc C$ as the following:
	\begin{align*}
		\mc C [f] := \mc C^+[f] -  \Lambda_\gamma (v) f,
	\end{align*} where $\mc C^+$ is given by 
	\begin{align*}
		\mc C^+ [f] := \int_{\R^d} \int_{\S^{d-1}} |v-v_*|^\gamma b(\cos \theta )  f' M (v_*')\d \sigma \d v_*,
	\end{align*} and the collision frequency $\Lambda_\gamma$ is given by 
	\begin{align*}
		\Lambda_\gamma (v): =  
		\int_{\R^d} \int_{\S^{d-1}} |v-v_*|^\gamma b(\cos \theta ) \d \sigma \d v_* = \int_{\R^d} |v-v_*|^\gamma M(v_*) \d v_*,
	\end{align*} with $\Lambda_\gamma (v ) \geq 0$ and it satisfies, for all $v \in \R^d$, that
	\begin{align*}
		0 \leq \Lambda_\gamma (v ) \leq (1+ |v|^2)^{\frac{\gamma}{2}}. 
	\end{align*}
	\item \textbf{Carleman representation.} Another representation is due to Carleman and it consists in choosing the pre-collisional velocities as new variables $v'$ and $v_*'$ from the set of admissible velocities. 
	This means 
	\begin{align*}
		\mc C^+ [f] &= \int_{\R^d} \frac{f'}{|v-v'|^{d-1}} \int_{E_{(v,v')}} B (|u|, r) M (v_*') \d v_*' \d v' 
		\\ &= \int_{\R^d} \frac{f'}{|v-v'|} \int_{E_{(v,v')}} |2v - v' - v_*'|^{\gamma - d -2} M (v_*') \d v_*' \d v',
	\end{align*} where the hyperplane $E_{(v, v')} : = \{v_*' \in \R^d \mid (v-v') \cdot (v-v_*') = 0\}$ is the set of admissible velocities. The second line comes from the identity $u := v- v_* = 2v - v' - v_*'$ and the assumptions on $B$ which essentially mean that $B$ can be written as $B(|u|, r ) = C |u|^\gamma r^{d-2}$ where $r = |v-v'|/|u|$. 
	\item \textbf{$\omega-$representation.} We sometimes use the so-called $\omega-$representation for the collisions, which is given by 
	\begin{align*}
		v' = v - (u \cdot \omega)\omega ,  \quad \mbox{and} \quad  	v_*' = v_* - (u \cdot \omega) \omega. 
	\end{align*} Then we have the following relation: $ \sigma = u/ |u| - 2 (u/ |u| \cdot \omega) \omega$. We re-write the Boltzmann collision kernel with this change of variables 
	\begin{align*}
		\tilde B (|v-v_*|, | \tilde r |) = |v- v_*|^\gamma \tilde b (|\tilde r |), \quad \tilde r = u/|u| \cdot \omega. 
	\end{align*} Note that $\tilde B$ and $\tilde b$ are different than $B$ and $b$ in the $\sigma-$representation. 
\end{enumerate}

Concerning the long-time behaviour of the linear Boltzmann equation, convergence to equilibrium in weighted $L^2$ spaces can be shown following \cite{DMS15, MN06}. A recent quantitative result which uses Harris-type theorems is \cite{CCEY20}. The main result of \cite{CCEY20} is given by
\begin{theorem}
	Suppose that $t \mapsto f_t$ is the solution of the linear Boltzmann equation \eqref{eq:kinetic}-\eqref{eq:linearBoltzmann} with the initial data $f_0 \in \mc P (\O \times \mc V)$. We consider the Boltzmann collision kernel with \eqref{eq:Boltzmannkernel} for some $\gamma \geq 0$, and assume that $b$ is integrable and uniformly positive on $[-1,1]$, i.e. there exists a positive constant $C_b$ such that $b (r) \geq C_b$ for all $r\in [-1,1]$. 
	We consider either of the three regimes below:
	\begin{itemize}
		\item[\textbf{R1.}] Suppose that $\O\times \mc V = \T^d \times \R^d$ (no confining potential, i.e. $\Phi =0$).
		\item[\textbf{R2.}] Suppose that $\O\times \mc V = \R^d \times \R^d$ and that the confining potential $\Phi \in C^2 (\R^d)$ satisfies that
		\begin{align}
			x \cdot \nabla_x \Phi \geq \alpha \langle x \rangle^{\gamma+2}  +\beta \Phi (x) - \eta,
		\end{align} for some positive constants $\alpha, \beta, \eta$.
	\end{itemize} 
	Then there exist positive constants $C, \lambda$ which are independent from $f_0$, such that 
	\begin{align}
		\|f_t - f_\infty\|_\phi \leq C e^{-\lambda t} \|f_0 - f_\infty\|_\phi,
	\end{align} where $f_\infty (x,v)= M(v) e^{-\Phi(x)}$ is the only equilibrium state and the
	weight function $\phi= 1+ \varphi$ where $\varphi \in \{\varphi_1, \varphi_2\}$ is given by 
	\begin{align}
		\varphi_{1}(v) &= H(x,v), \tag{for \textbf{R1}} \\
		\varphi_{2}(x,v) &= H(x,v) +|x|^2, \tag{for \textbf{R2}}
	\end{align} where $H(x,v) := \Phi(x) + |v|^2/2$ is the total energy of the system.
	\begin{itemize}
		\item[\textbf{R3.}] Suppose that $\O\times \mc V = \R^d \times \R^d$ and that the confining potential $\Phi \in C^2 (\R^d)$ satisfies, for $\xi \in (0,1)$ that
		\begin{align*}
			x \cdot \nabla_x \Phi \geq \alpha \langle x \rangle^{1+\xi} + \beta \Phi (x) - \eta \quad \mbox{and} \quad \Phi \leq \delta \langle x \rangle^{1+\xi}.
		\end{align*} for some positive constants $\alpha, \beta, \eta, \delta $.
	\end{itemize}
	Then there exists a positive constant $C$ such that 
	\begin{align}
		\|f_t - f_\infty\|_1 \leq \min \left\{ \|f_t - f_\infty\|_1, C (1+t)^{-\xi} \|f_0\|_\phi \right\},
	\end{align} where $f_\infty (x,v)= M(v) e^{-\Phi(x)}$ is the only equilibrium state and $\phi = 1 + H(x,v) + |x|$.
\end{theorem}
\begin{proof}
	We start with an initial data $f_0 (x,v)= \delta_{(x_0,v_0)}(x,v)$ satisfying $H(x_0, v_0) \leq E_0$ ($H$ is the total energy or Hamiltonian) for some $E_0 >0$ then we can find a constant $C>0$ such that (similar to \eqref{firstDuhamel})
	\begin{align} \label{Duhamel2}
		f_t \geq e^{-Ct}\int_0^t \int_0^s \mc T_{t-s}  \mc  {\tilde C}^+ \mc T_{s-r} \mc  {\tilde C}^+ \mc T_r (\1_E f_0 ) \d r \d s,
	\end{align} where ${\mc{\tilde C}}^+  \nu := \1_E \mc C^+ \nu$ for all $ \nu \in \mc P $ and the set $E : = \{ (x,v) \in \O \times \mc V \mid H(x,v) \leq E_0\}$. 
	
	For the bound on the action of $\mc C^+$, for any $\beta, \tilde \beta  >0$, we can find an $\varepsilon >0 $ such that for any probability measure $\nu \in \mc P (\T^d \times \R^d)$ and for any $v \in \R^d$ with $|v| \leq \tilde \beta $ the following holds:
	\begin{align} \label{ineq}
		\mc C^+ \nu (x,v) \geq \varepsilon \int_{B(0, \beta )} \nu (x,v') \d v'.
	\end{align}
	Now using these two arguments we give the idea of the proofs in three different regimes below. 
	\paragraph{\textbf{ \em R1.}} Unlike the linear BGK equation in the regime \textbf{\em R1}, here on the torus, we cannot get a uniform lower bound since the equation exhibits less global behaviour in $v$. However, combining the argument used in the proof of  \textbf{\em R1} for the linear BGK equation and \eqref{ineq}, we show that the \ref{Minorisation} holds on the bounded sets in $|v|$. 
	
	The \ref{Foster_Lyap} holds with $\phi (v) = |v|^2$. To see that we look at the actions of the operators $\mc T$ and $\mc C$ on $|v|^2$ together. Note that the transport part plays no role, i.e. $ \int \mc T [f] |v|^2 \d z =0$. For the collision part, we use the $\omega-$representation of the velocities and check the action of the operator with this formulation on $|v|^2$ and obtain the result.
	\paragraph{\textbf{ \em  R2.}} The \ref{Minorisation} in this case is very similar to the regime \textbf{\em  R2} of the linear BGK equation and it is obtained by using \eqref{Duhamel2}, \eqref{ineq} and \eqref{ineq3}. 
	
	For the \ref{Foster_Lyap} we look for a functional of the form (with $\alpha, \beta >0$)
	\begin{align*}
		\phi (x,v) = H(x,v) + \alpha x \cdot v + \beta |x|^2,
	\end{align*} and we want to show that for appropriate choices of $\alpha$ and $\beta$, $\phi$ satisfies \eqref{eq:FL3} for some $\zeta, D >0$. Note that $\mc L^* =\mc C^*- \mc T^*$ in this case (where the linear Boltzmann equation is written as $\p_t f = \mc L [f]$). We look at the actions of these operators individually on each term of $\phi$, i.e. we can show for appropriate positive constants $\alpha_1, \alpha_2$, that
	\begin{align*}
		\mc C^* (\Phi (x)) &= 0, \quad   \mc C^* (|v|^2) \leq - \alpha_1 \langle v \rangle^{\gamma+2} + \alpha_2, \quad  \mc C^* (x \cdot v) \leq   \langle v \rangle^{\gamma+1} x, \\
		\mc T^* (H(x,v)) &= 0, \quad  \mc T^* (|x|^2) = 2 x \cdot v, \qquad  \qquad \quad \mc T^*(x \cdot v) = -x \cdot \nabla_x \Phi (x) + |v|^2.
	\end{align*} Combining all these we see that $\phi$ satisfies \eqref{eq:FL3}. Then the result follows from Theorem \ref{thm:harris}.
	\paragraph{\textbf{ \em R3.}} The \ref{Minorisation2} is the same as \textbf{\em R2} above. 
	
	For the \ref{Foster_Lyap2} we have to pay more attention on how $\mc C$ acts on the $x \cdot v$ moment since the confining potential is not strong enough to verify the \ref{Foster_Lyap}. We can find $\alpha, \beta >0$ ,  $4\alpha^2 < \beta$ such that $(\mc C^*-\mc T^*) \phi \leq -\zeta V(\phi ) + D $ is satisfied with 
	\begin{align*}
		\phi (x,v) = H(x,v) + \alpha \frac{x \cdot v}{\langle x \rangle } + \beta \langle x \rangle, \quad \mbox{and} \quad V(s) = 1 + s^{\frac{\xi}{1+\xi}}.
	\end{align*} This verifies the \ref{Foster_Lyap}. Thus Theorem \ref{thm:subgeoHarris} gives the result. 
\end{proof}

\subsubsection{The linear degenerate Boltzmann equation}

\label{sec:deg_Boltzmann}
In this section, we consider the linear degenerate Boltzmann equation, i.e. the collision operator $\mc C$ in \eqref{eq:kinetic} takes the form 
\begin{align} \label{eq:degBoltzmann}
	\begin{split}
		\mc C [f] &= \sigma (x) \int_{\mc V} \left( B (v, v') f ' - B (v',v) f\right) \d v',
	\end{split}
\end{align} where $B$ is the scattering kernel and $\sigma \in L^\infty(\O)$ is a non-negative function. This equation is used as a model for a radiative transfer system where, depending on the scattering kernel $B$, the energy is distributed differently among different parts of the domain. The equilibrium state of \eqref{eq:kinetic} with \eqref{eq:degBoltzmann} is given by 
\begin{align} \label{eq:equilib_linBGK}
	f_\infty (x,v) = \frac{1}{Z} e^{-\Phi(x)},
\end{align} where $Z = |\mc V| \int_{\O} e^{-\Phi(x)} \d x \footnotemark$ if $B(v,v') = 1/ |\mc V|$ or $Z = M(v) \int_{\O} e^{-\Phi(x)} \d x$ where $M(v)$ is the normalised Maxwellian \eqref{Maxwell}. 
\footnotetext{The volume of the set $\mc V$ is denoted by $|\mc V|$.}

The so-called \emph{non-degenerate} case (where $\sigma >0$) is well-studied using different hypocoercivity tools. Notice that if we consider $\sigma =1$ and $B$ a diffusive scattering kernel given by a Maxwellian velocity distribution then we recover the linear relaxation Boltzmann (or the linear BGK) equation \eqref{eq:linBGK} and the convergence to equilibrium for this case was the subject of Section \ref{sec:linearBGK}. 

In the \emph{degenerate} case (where $\sigma \geq 0$) the convergence to equilibrium depends on the phase space $\O \times \mc V$ and on the geometry of the set $\{\sigma = 0\}$. This makes the study more delicate. It was shown that if the support of $\sigma$ satisfies a condition called \emph{Geometric Control Condition} (GCC), then it is possible to show exponential convergence to the equilibrium (see \cite{HKL15, BS13, BS13_2} and references therein for a more detailed discussion). The inspiration for using this type of condition in the context of the Boltzmann-type equations comes from earlier results in the study of controllability and stabilisation properties of hyperbolic PDEs, particularly the wave equation.

In this section, we look at the main result of \cite{EM19} where the authors gave the first quantitative exponential convergence rates by using Doeblin's theorem for the degenerate linear Boltzmann equation when $x \in \O = \T^d$. 

\begin{theorem} \label{thm:degenerateBolztmann}
	Suppose that $t \mapsto f_t$ is the solution of the linear degenerate Boltzmann equation \eqref{eq:kinetic}-\eqref{eq:degBoltzmann} with the initial data $f_0 \in \mc P (\T^d \times \mc V)$. We consider either of the two regimes below: 
	\begin{itemize}
		\item[\textbf{R1.}]  \label{DB1}
		Suppose that $\O \times \mc V = \T^d \times \mc V$ where $\mc V \subseteq \R^d$ is a bounded, open set, $\Phi =0$ and the following two conditions are satisfied:
		\begin{itemize}
			\item[\textbf{A1.}] There exists $\gamma >0$ such that 
			$B(v,v') \geq \dfrac{\gamma}{|\mc V|}$ for all $v,v' \in \mc V$.
			\item[\textbf{A2.}] Let $\sigma \in C^0 (\mathbb{T}^d)$ and there exist two constants $T = T(\sigma) >0$ (depending on $\sigma$), $\kappa >0$ such that 
			\begin{align*} \tag{GCC 1}
				\inf_{(x,v) \in \T^d \times \mc V} \int_0^T \sigma (x+vt) \d t \geq \kappa. 
			\end{align*}
		\end{itemize}
		\item[\textbf{R2.}]   \label{DB2}  
		Suppose that $\O \times \mc V = \T^d \times \R^d$, $\Phi \in C^2(\T^d)$ and the following two conditions are satisfied:
		\begin{itemize}
			\item [\textbf{A1.}] There exists a strictly positive decreasing function $b$ such that $B(v,v') \geq b (|v|)$ for all $v,v' \in \R^d$.
			\item [\textbf{A2.}] Let $\sigma \in C^0 (\mathbb{T}^d)$ and there exist two constants $T = T(\sigma, \Phi) >0$  (depending on $\sigma$ and $\Phi$), $\kappa >0$ such that 
			\begin{align} \label{GCC2}
				\tag{GCC 2}
				\inf_{(x,v) \in \T^d \times \R^d} \int_0^T \sigma (X_t(x,v)) \d t \geq \kappa,
			\end{align} where $(X_t(x, v), V_t(x,v))$ solves \eqref{eq:ODE1} for any $(x,v) \in \T^d \times \mc \R^d$.
		\end{itemize}
	\end{itemize}
	Then there exist positive constants $C, \lambda$ such that 
	\begin{align} \label{convDEG}
		\|f_t - f_\infty\|_{1} \leq C e^{-\lambda t} \|f_0 - f_\infty\|_{1},
	\end{align} where $f_\infty$ is the only equilibrium state given by \eqref{eq:equilib_linBGK}.
\end{theorem}

\begin{proof}[Idea of the proof of Theorem \ref{thm:degenerateBolztmann}]
	The proof uses Doeblin's theorem. \ref{Doeblin} means that we quantify the effect of the transport operator on the phase space which satisfies specific assumptions, i.e. GCC, on its geometry. We present the sketch of the proof in the two regimes together.
	
	Starting from an arbitrary point with $(x_0, v_0) \in \T^d \times \mc V$ in both regimes it is possible to find some time $ t^*>0$ and a constant $\beta \in (0,1)$ such that 
	\begin{align} \label{ineq1}
		\inf_{x \in \T^d} \int_{\mc V} \mc T_t (\delta_{x_0} (x) B(v, \cdot)) \d v \geq \frac{\beta}{|\T^d|},
	\end{align} is satisfied for all $t \geq  t^*$ in the case of \textbf{\em R1} and for $t \in [ t^*,  t^* + T]$ in the case of \textbf{\em R2} (where $T$ comes from \eqref{GCC2}). In the latter case, the argument is iterated by using the estimate for short time after considering transport for a long time,  i.e. allowing jumps from a high velocity to other high velocities.

	More precisely, if we take 
	\begin{itemize}
		\item 	$ t^* = r_0/2$ and $\beta = \gamma (r_0/2)^d$ ($r_0$ is such that $B(v',v) \geq \gamma \1_{\{ v \in B(v_0, r_0)\}}$) for \textbf{\em R1},
		\item $ t^* = 1/2$ and $\beta = e^{-C(\|\nabla \Phi\|_\infty, T) }\int_{\T^d} e^{-\Phi(x)} \d x$ for \textbf{\em R2},
	\end{itemize} then we see that \eqref{ineq1} is verified. Note that $C(\|\nabla_x \Phi\|_\infty, \hess \Phi, b, T)$ is a positive constant depending on $\|\nabla_x \Phi\|_\infty, \hess \Phi, b$ and $T$. 
	
	We take $f_0 \in \mc P (\T^d \times \mc V)$ and we would like to find a positive lower bound for the solutions for $f_t$ at time $t \geq \tau$. Similar to \eqref{firstDuhamel}, for the solutions of the degenerate linear Boltzmann equation we have
	\begin{align*}
		f_t \geq e^{-t\|\sigma\|_\infty}\int_0^t \int_0^s \mc T_{t-s} m_\sigma  \mc C^+ \mc T_{s-r} m_\sigma \mc C^+ \mc T_r f_0 \d r \d s.
	\end{align*} where $m_\sigma \mu := \sigma(x) \mu$ and $\mc C^+ \mu = \int_{\mc V} B(v,v') \mu (\d v)$ for any $\mu \in \mc P$. Combining this with \eqref{ineq1} we obtain for $t \geq t^* + 2T$ 
	\begin{align*}
		f_t \geq \beta \kappa^2 e^{-t \|\sigma \|_\infty} \nu, \quad \mbox{for all } \nu \in \mc P. 
	\end{align*} Therefore \ref{Doeblin} is verified with $\tau = t^* + 2T$ and $\alpha = \beta \kappa^2 e^{-t \|\sigma \|_\infty}$.  	This concludes the proof. 
\end{proof} 
\begin{remark}
	We can compute $C, \lambda$ in \eqref{convDEG} explicitly following Theorem \ref{thm:Doeblin}. In this case,
	\begin{align*}
		C = \frac{1}{1- \beta \kappa^2  e^{-\tau  \|\sigma\|_\infty}}, \quad \mbox{and} \quad \lambda = - \frac{\log (1- \beta \kappa^2  e^{-\tau \|\sigma\|_\infty})}{\tau }.
	\end{align*} 
\end{remark}

\subsubsection{The linear kinetic Fokker-Planck equation}

\label{sec:FP}
In this section, we consider the kinetic Fokker-Planck equation with the general force term and it is given by 
\eqref{eq:kinetic} where the the collision operator $\mc C$ is defined as 
\begin{align} \label{eq:kinetic_FP}
	\mc C [f] &= \nabla_v \cdot (\nabla_v f + \nabla_v\Psi(v)f) = \Delta_v f + \nabla_v \cdot ( \nabla_v\Psi(v)f),
\end{align} where the confining potential $\Phi$ and the friction term $\Psi$ satisfy
\begin{align} \label{kinFokPlanck}
	\Phi (x) = \frac{\langle x\rangle^\gamma}{\gamma},  \quad \mbox{and} \quad \Psi(v) = \frac{\langle v\rangle^\beta}{\beta},  \, \beta \geq 2.
\end{align} We consider the phase space $\O\times \mc V = \R^d \times \R^d$. 
When $\beta = 2$ the resulting equation is known as the \emph{classical} kinetic Fokker-Planck equation whose global equilibrium state is given by the so-called Gibbs distribution
\begin{align} \label{eq:Gibbs}
	f_\infty (x,v)=  Z^{-1} e^{-\Phi (x) - \frac{|v|^2}{2}}, \quad  Z = \int e^{-\Phi (x) - \frac{|v|^2}{2}} \d z.
\end{align} In the classical case ($\beta = 2$) when $\gamma \geq 1$, the exponential convergence towards the global equilibrium \eqref{eq:Gibbs} is well-known, see e.g. \cite{HN04, V09, MM16} and references therein. Moreover, the well-posedness of the Cauchy problem is implied by the hypoelliptic nature of the equation, see \cite{HN04} and references therein.

For weaker confining potentials, i.e. when $\gamma \in (0,1)$, there is no spectral gap and the convergence is sub-geometric. In \cite{BCG08,DFG09}, the authors proved respectively that the convergence rates are polynomial and subgeometric in the weighted total variation norms. They used probabilistic arguments including Harris type theorems. Particularly in \cite{BCG08}, the authors showed how Harris-type arguments can be useful to obtain Poincaré inequalities for the kinetic Fokker-Planck equation for $\gamma \geq 1$. It is known that in this case, the classical Poincaré inequality does not hold. 

In \cite{C20}, the author generalised the previous results and provided the first quantitative algebraic rate in the case where $\gamma \in (0,1)$ and $\beta = 2$ by using classical hypocoercivity methods. Moreover in a later work \cite{C21}, the author considered the kinetic Fokker-Planck equation with general force term, i.e. \eqref{eq:kinetic} with \eqref{eq:kinetic_FP} with $\beta \geq 2$ and $\gamma \geq 1$ and shows the exponential convergence to a unique equilibrium in a quantitative manner by using Harris-type arguments. 

Lastly, we also mention \cite{L20} where the author used Harris-type approaches to show the existence of a stationary solution and convergence towards it for the fractional Fokker-Planck equation.  

In this section we present the main results of \cite{C20, C21} in the following theorem:

\begin{theorem} \label{thm:kinetic_Fokker_Planck}
	Suppose that $t \mapsto f_t$ is the solution of the kinetic Fokker-Planck equation \eqref{eq:kinetic}-\eqref{eq:kinetic_FP} with the initial data $f_0 \in \mc P_{\phi_2}^p (\R^d \times \R^d)$ with $p \in [1, +\infty)$. We consider either of the three regimes below:
	\begin{itemize}
		\item [\textbf{R1.}] Suppose that $\beta = 2$ and $\gamma \in (0,1)$ in \eqref{kinFokPlanck}.
		\item [\textbf{R2.}] Suppose that $\beta \geq 2$ and $\gamma \geq 1$ in \eqref{kinFokPlanck}.
	\end{itemize}
	Then there exist positive constants $C, \lambda$ which are independent from $f_0$, such that 
	\begin{align}
		\|f_t - f_\infty\|_{L^p_ {\phi_1}} \leq C e^{-\lambda \alpha (t)}  	\|f_t - f_\infty\|_{L^p_{\phi_2}}, 
	\end{align}
	where $f_\infty$, given by \eqref{eq:Gibbs}, is the only equilibrium state, and for small constants $\varepsilon, \delta >0$, the weight functions $\phi_1,\, \phi_2$, which are independent from $f_0$, are given by 
	\begin{alignat*}{5}
		\phi_1(z) &= f_\infty(z)^{- \frac{p-1}{p}}, \qquad &\phi_2(z) &= f_\infty (z)^{- \frac{p-1}{p}- \delta}, \quad &\alpha (t) &= t^\xi, \quad \tag{for \textbf{R1}} \\
		\phi_1 (z) &=\phi_2 (z)=  e^{\chi\varphi (z)}, \qquad &\varphi (z) &= |v|^2 + \Phi (x) + \varepsilon v \cdot \nabla_x \langle x \rangle, \quad &\alpha (t) &= t, \quad \tag{for \textbf{R2}}
	\end{alignat*} with $\chi>0$ and $\xi \in (0, \gamma/(2-\gamma))$.
	\begin{itemize}
		\item [\textbf{R3.}] Suppose that $\beta = 2$ and $\gamma \in (0,1)$ in \eqref{kinFokPlanck}. Then for all $f_0 \in \mc P_\phi (\R^d \times \R^d)$ there exists a positive constant $C>0$, and  
		$k \geq 1$ we have 
		\begin{align*}
			\|f_t - f_\infty\|_1 \leq C (1+t)^{- \frac{k}{1- \gamma/2}}	\|f_t - f_\infty\|_{\phi},
		\end{align*} where $\phi (z) = (1 +\varphi (z))^k$ with $\varphi (x,v) = |x|^2 + |v|^2$. 
	\end{itemize}
\end{theorem}

\begin{proof}
	[Idea of the proof of Theorem \ref{thm:kinetic_Fokker_Planck}] We only give the idea of the proof for the regime \textbf{\em R2} as this is the only part which uses Harris-type arguments.
	The regimes \textbf{\em R1, R3} use classical hypocoercivity approaches. We refer to \cite{C20} for their proofs.
	\paragraph{\em R2.} The proof uses Theorem \ref{thm:harris}. We write the kinetic Fokker-Planck equation as $\p_t f = \mc L [f] =  (\mc C  - \mc T )[f]$ where $\mc C$ is given by \eqref{eq:kinetic_FP}. Then the \ref{Foster_Lyap} is verified with $\phi (z) := e^{\chi \varphi(z)}$ for $\beta>0$ and $\varphi (z) = |v|^2 + \Phi(x) + \ve v \cdot \nabla_x \langle x \rangle $,
	\begin{align*}
		\mc L^* \phi (z) \leq -\zeta  h(z)\phi (x) + D,
	\end{align*} where $h(z) = \langle v \rangle ^\beta + \langle x \rangle^{\gamma-1} + 1$. Here we proved Inequality \eqref{eq:FL3} which uses $\mc L^*$, the formal adjoint of $\mc L$. 
	
	The \ref{Minorisation} is verified by tracking the characteristics of the flow of the equation as in the previous sections. It is done in a couple of steps. First, we prove that the solutions at an initial point $(x_0,v_0) \in \O \times \mc V$ remain bounded from below by a constant. Then by some regularity estimates, we prove that the solutions are locally continuous. This enables us to extend the lower bound to a small ball from any initial point. Then we prove the \ref{Minorisation} by iterating this and applying Duhamel's formula. The details can be found in Theorem 5.2 of \cite{C21}. 
\end{proof}

\subsection{Kinetic equations arising in mathematical biology}
\label{sec:MathBio}

The Harris-type theorems have been successfully used in the study of quantitative convergence to equilibrium for several macroscopic equations arising in mathematical biology, particularly in structured population dynamics. Some important examples are the renewal equation \cite{G18}, structured neuron population models \cite{CY19, DG20, PS19, TPS21} and the growth fragmentation equation \cite{CGY21}. However, applications of these theorems on the mesoscopic models, so-called kinetic equations, in the context of biology are more recent. 

Kinetic equations obtained in the mean-field limit of the agent-level interactions in applied sciences differ from the kinetic equations arising in mathematical physics in some aspects. The former often do not have explicit stationary states, whereas, for a large class of kinetic equations in physics, the equilibrium state is given by a Gibbs (or in some cases Maxwellian) distribution. So, they can be written down explicitly. The classical hypocoercivity techniques often require obtaining Poincaré inequalities in the norms which are the inverse of the invariant measure, i.e. equilibrium, and this creates complications when the equilibrium is not known explicitly. For a more detailed discussion we refer to \textbf{Motivation and aim} in Section \ref{sec:introduction}.

Another difference is that for the kinetic equations arising in applied sciences the confinement mechanism is not induced via an external force field. It comes from the internal dynamics. This will be made more precise in the next section. Since the classical hypocoercivity methods are developed for the confinement mechanisms coming from a confining potential, sometimes they give restrictive results on the kinetic models used outside the gas theory.

In this section, we give two examples of how Harris-type theorems can be used efficiently for obtaining quantitative convergence results under relaxed assumptions and allow generalisations of the previous results which are based on classical hypocoercivity/general relative entropy techniques. Namely, we consider the run and tumble equation and the kinetic FitzHugh-Nagumo equation.

\subsubsection{The run and tumble equation} 

\label{sec:runtumble}

The run and tumble equation is a kinetic-transport equation which describes the movement of bacteria under the effect of a chemotactic substance. The movement of bacteria is a combination of a transport with a constant speed along the gradient of the chemoattractant, called \emph{run}, and a random change in the orientation towards the chemoattractant-dense regions, called \emph{tumble}. The model is introduced in \cite{S73,A80} based on some experiments done on a bacterium called \emph{E.coli}. The operators $\mc T$ and $\mc C$ in \eqref{eq:kinetic} are given by
\begin{align} \label{eq:RT}
	\begin{split}
		\mc T [f] (t,z) &= v \cdot \nabla_x f (t,z), \\
		\mc C [f] (t,z) &= \int_{\mc V} \lambda (v' \cdot \nabla_x M(x)) \kappa (v,v') f' \d v' - \lambda (v \cdot \nabla_x M(x)),
	\end{split}
\end{align} where $f := f(t,x,v) \geq  0$ (resp., $f^\prime : = f(t,x,v^\prime)$) is the probability density of finding a bacteria at time $t \geq 0$, at the position $x \in \R^d$, moving with a velocity $v \in \mc V \in \R^d$ where $\mc V$ is a unit ball centered at the origin with a radius $R_0$, i.e. $\mc V = B (0, R_0)$, such that $|\mc V| = 1$. Here $v$ and $v'$ denote the pre-and post-tumbling velocities respectively. In the definition of the operator $\mc C$, $\lambda : \R \to [0, + \infty)$ is called the tumbling rate. We define the external signal $M(x)$ in terms of the chemoattractant concentration $S(x)$: $M(x) = \log (S(x))$. The tumbling kernel $\kappa (v,v')$ gives the probability distribution of moving from velocity $v$ to velocity $v'$ so that it satisfies that $\int_{\mc V} \kappa (v,v') \d v' =1$. For existence results on the run and tumble equation we refer to \cite{BC10} and references therein. 

In \cite{MW17}, the authors showed the existence of a nontrivial stationary state for this equation and exponential convergence towards it when $\O= \R^d$. They assume that the chemoattractant density function $S(x)$ is radially symmetric. This assumption essentially reduces the result to dimension $d=1$. 

A recent quantitative result concerning the hypocoercivity of this equation using Harris's theorem can be found in \cite{EY21} where the authors could remove the radial symmetry assumption on $S(x)$ of \cite{MW17}. 

The main result of \cite{EY21} is given by

\begin{theorem} \label{thm:RT}
	Suppose that $t \mapsto f_t$ is the solution of the run and tumble equation \eqref{eq:kinetic}-\eqref{eq:RT} with the initial data $f_0 \in \mc P (\R^d \times \mc V)$. Suppose further that the following hypotheses hold:
	\begin{itemize}
		\item[\textbf{A1.}] The tumbling kernel is uniform, i.e. $\kappa \equiv 1$. 
		\item[\textbf{A2.}]  \label{RT1} The tumbling rate $\lambda : \R \to [0, + \infty)$ is given by $\lambda (m) = 1 - \chi \psi (m)$ where $\chi \in (0,1)$ is the chemotactic sensitivity and $\psi$ is either $\psi (m) = \sign (m)$ or it is a bounded, odd and increasing function with $\|\psi\|_\infty \leq 1$. Suppose also that $\psi (m)m$ is differentiable. 
		\item[\textbf{A3.}] \label{RT2} There exist $R\geq 0$ and $\tilde C >0$ such that whenever $|x|>R$, $|\nabla_x M(x)| \geq \tilde C$. Moreover $|\nabla_x M|$, $\hess (M)$ are bounded and 
		\begin{align*}
			\lim_{|x| \to \infty}  M (x) = - \infty, \quad 	\lim_{|x| \to \infty}  \hess (M) (x) = 0.
		\end{align*}
		\item[\textbf{A4.}] \label{RT3} There exist a positive constant $\tilde \lambda (\psi, \|\nabla_x M\|_{\infty}) $ (depending on $\psi$ and $\|\nabla_x M\|_{\infty}$) and a positive integer $k(\psi)$ (depending on $\psi$) such that 
		\begin{align}
			\int_{\mathcal{V}} 	\psi(m') m' \d v'   \geq \tilde{\lambda} (\psi, \|\nabla_x M\|_{\infty}) |\nabla_xM(x)|^k,
		\end{align}  where $m' = v' \cdot \nabla_x M(x)$.
	\end{itemize}
	Then the run and tumble equation admits a unique stationary state, $f_\infty$,  and there exist positive constants $C, \Lambda$ which are  independent from $f_0$, such that for all $t \geq0$, 
	\begin{align} \label{eq:thmRT}
		\|f_t - f_{\infty} \|_\phi  \leq C e^{-\Lambda t} \| f_0 -f_{\infty}\|_{\phi}, 
	\end{align} where \begin{align} \label{FL:runtumb1}
		\phi (x,v) = \left(  1 -\gamma m -\beta \gamma \psi(m) m \right)e^{-\gamma M(x)} 
	\end{align}  with $m = v \cdot \nabla_x M(x)$, and $\gamma, \beta$ positive constants which can be computed explicitly. Moreover, suppose that 
	\begin{itemize}
		\item[\textbf{A4.}] There exist positive constants $C_1,  C_2,\alpha $ such that \begin{align*}
			C_1 -\alpha \langle x \rangle \leq M (x) \leq C_2 - \alpha \langle x \rangle,
		\end{align*} where $\langle x \rangle = \sqrt{1+ |x|^2}$.
	\end{itemize}
	Then the contraction \eqref{eq:thmRT} holds (for different positive constants $C$ and $\lambda$) with an exponential weight function given by
	\begin{align*}
		\phi (x) = e^{\omega \langle x \rangle}.
	\end{align*} where $\omega$ is a negative constant.
\end{theorem}
\begin{proof}[Idea of the proof of Theorem \ref{thm:RT}]
	The proof uses Harris's theorem. The \ref{Foster_Lyap} is satisfied with the weight function \eqref{FL:runtumb1} and with the constants
	\begin{align*}
		\beta = \frac{\chi}{1+ \chi}, \quad \gamma \leq \min \left \{ \frac{\tilde{\lambda}\chi(1-\chi)\xi}{8(1+\chi)} , \frac{1+\chi}{2(2+\chi)R_0 \|\nabla_x M\|_\infty} \right \}, \quad \mbox{with} \quad \xi:=
		\begin{cases}
			\tilde C^{k-2}, \quad &\mbox{if } k<2,\\ 
			1, \quad &\mbox{if } k=2, \\
			\|\nabla_x M\|_\infty^{k-2}, \quad &\mbox{if } k>2.
		\end{cases}
	\end{align*} All the constants appearing above are coming from the hypotheses. The computations can be found in Section 2.1 of \cite{EY21} but here we discuss how to get an intuition about the shape of the weight function \eqref{FL:runtumb1}. 
	
	We expect that Lyapunov functional will have exponential tails if we would like to have an exponential decay to equilibrium. We also notice that the confinement terms are bounded in  \eqref{eq:RT}. Therefore, we choose a function of $M$ which is $e^{-\gamma M}$ and it behaves like our guess. Then we look for a weight function which is closely related to $e^{-\gamma M}$. We repeatedly differentiate $\int_{\R^d \times \mc V} f (t,z) \d z$ along the flow of the equation which we re-write as $\p_t f  = \mc L f$. This means  that we apply the formal adjoint $\mc L ^*$, of $\mc L$, to the terms of the function $\phi$ until there is a term which does not change sign. Then the precise form of \eqref{FL:runtumb1} is found by combining $e^{-\gamma M}$ and the terms appearing in the derivatives of $\int_{\R^d \times \mc V} f (t,z) \d z$.
	
	To verify the \ref{Minorisation}, we first consider the transport equation  $\p_t f + v \cdot \nabla_x f+ \lambda (x,v) f = 0$ and the jump operator $\mc J [f] = \int_\mc V \lambda' f' \d v'$ where $\lambda' := \lambda (v \cdot \nabla_x M(x))$. We write the solution to the transport part by the method of characteristics. Then we apply iteratively the transport and the jump operators to the Dirac measure intial data $f_0 (x,v) = \delta_{(x_0,v_0)}(x,v)$ where $ (x_0, v_0) \in \R^d \times \mc V$ and give a lower bound for the solution of the full equation by re-writing it using Duhamel's formula. After each jump we integrate the transported measure in $\mc V$. Finally we obtain the \ref{Minorisation} by choosing $\tau = 3 + \tilde R/ R_0$ for all $x_0 \in B (0, \tilde R) $ with an explicit bound depending on $\tilde R, R_0, \chi$, and independently from the initial data. Since the underlying semigroup is Markov, we extend the initial data to any probability measure. 
	Having these two conditions satisfied, Theorem \ref{thm:harris} gives the result.
	
	In order to show the contraction in the norm with an exponential weight, we use that fact that the conditions on $\gamma$ are enough to obtain the inequality $e^{-\gamma M} \leq 2 \phi$ where $\phi$ is given by \eqref{FL:runtumb1}. Combining this with the bound on $M$, we can find a constant $\omega  <0$ such that the contraction \eqref{eq:thmRT} holds in the weighted total variation norm with the weight $e^{\omega \langle x \rangle } $. 
\end{proof}

As in the case above, when $S$ is a fixed function of $x$, the resulting equation is linear. However, in more realistic models of run and tumble process, $S$ solves a screened Poisson equation, i.e. 
\begin{align} \label{eq:S-Poisson}
	- \Delta_x S + \alpha S = \rho(t,x) := \int_{\mathcal{V}} f(t,x,v) \d v.
\end{align} It is difficult to prove even the existence of a stationary state for the problem when the run and tumble equation is coupled with \eqref{eq:S-Poisson}. We refer to e.g. \cite{C20_2} where the author showed the existence of travelling wave solutions for the coupled nonlinear run and tumble equation and some other related works can be found in the references of \cite{C20_2}. 

In \cite{EY21}, we consider another nonlinear coupling (given in \eqref{weak_nonlin} below) for the run and tumble equation. This coupling plays the role of being an intermediate step between the linear model \eqref{eq:kinetic}-\eqref{eq:RT} and the nonlinear model \eqref{eq:kinetic}-\eqref{eq:RT}-\eqref{eq:S-Poisson}. A detailed discussion about the motivation behind \eqref{weak_nonlin} and how it is obtained can be found in the last section of \cite{EY21}.  The result concerning this nonlinear equation is given by
\begin{theorem} \label{thm:weak_NL}
	Suppose that $t \mapsto f_t$ is the solution of the run and tumble equation \eqref{eq:kinetic}-\eqref{eq:RT} with an initial data $f_0 \in \mc P (\R^d \times \mc V)$. Suppose further that \textbf{A1, A2} and \textbf{A3} of Theorem \ref{thm:RT} hold and that $\psi$ is a Lipschitz function. 
	\begin{itemize}
		\item[\textbf{A4.}] We assume also that $S$ solves
		\begin{align} \label{weak_nonlin}
			S (x) = S_\infty (x) (1+ \eta N * \rho ),
		\end{align} where $N$ is a positive, smooth function with a compact support, and $\eta>0$ is a small constant.
		\item[\textbf{A5.}]  Let $S_\infty$ be a smooth function satisfying for some positive constants $C_1, C_2, \alpha$ that
		\begin{align*} 
			C_1 -\alpha \langle x \rangle \leq M_\infty (x): =\log(S_\infty(x))\leq C_2 - \alpha \langle x \rangle.
		\end{align*}
	\end{itemize} Then there exists some constant $\tilde{C}$ depending on $C_1, C_2,$ and $\alpha$ such that if $\eta < \tilde{C}$ then there exists a unique steady state solution to the nonlinear run and tumble equation \eqref{eq:kinetic}-\eqref{eq:RT} where $S$ solves \eqref{weak_nonlin}. Moreover for any initial data $f_0 \in \mathcal{P} (\R^d \times \mathcal{V})$ satisfying $\|f_0\|_\phi \leq \bar C$ with $\phi = e^{\omega \langle x \rangle}$ where $\omega <0$ and the constant $\bar C>0$  depends on $\|\psi'\|_\infty, \|\nabla_x N\|_\infty$ and some other constants coming from the previous  hypotheses and Theorem \ref{thm:RT}; there exist some positive constants $C$ and $\Lambda$ satisfying
	\begin{align*}
		\| f_t - f_\infty\|_\phi \leq Ce^{-\Lambda t/2} \|f_0 -f_\infty\|_{\phi},
	\end{align*} where $\phi = e^{\omega \langle x \rangle}$.
\end{theorem}

\begin{proof}
	[Idea of the proof of Theorem \ref{thm:weak_NL}] 
	First, we carry out a fixed point argument to show that there exists a unique stationary state for the nonlinear equation. More precisely, we define a function $\mc G (M): C^2 (\R) \to C^2 (\R)$ given by 
	\begin{align*}
		\mc G (M) =\log (S_\infty (1 + \eta N * \rho^M)),
	\end{align*} where $S_\infty, \eta, N$ are the same as in \eqref{weak_nonlin} and $\rho^M (x) = \int_{\mc V} f_\infty (x,v) \d v$ where $f_\infty$ is the equilibrium state for the linear run and tumble equation with the choice $M$ in $\mc C$. This means that if $M$ is a fixed point of $\mc G$, then $f_\infty$ will be a steady state of the nonlinear equation. We can show that $\mc G$ has a unique fixed point by using the contraction mapping theorem, i.e. we obtain an inequality for $M_1, M_2$ such that $M_1 \neq M_2$,
	\begin{align*}
		\|\mc G (M_1) - \mc G (M_2)\|_{W^{1, \infty}} \leq 	\chi \|M_1 - M_2 \|_{W^{1, \infty}}
	\end{align*} where $\chi \in (0,1)$ is a constant and $W^{1, \infty}$ is the Sobolev space with $n=1$ and $p = \infty$ (see Section \ref{sec:Harris} for the definition).
	Next, we use the fact that the assumptions made on $N$ and $\eta$ in \eqref{weak_nonlin} enable us to consider the nonlinear equation as a perturbation of the linear one. The final convergence estimate is then obtained by exploiting these two arguments and using some of the estimates coming from the proof of Theorem \ref{thm:RT}.
\end{proof}

\subsubsection{The kinetic FitzHugh-Nagumo equation} \label{sec:FitzHugh-Nagumo}

The kinetic FitzHugh-Nagumo equation is a nonlocal PDE describing the dynamics of interacting neurons structured with two variables: the membrane potential (or voltage) $v$ and the adaptation (or recovery) variable $x$. The model is proposed in \cite{NAY62,FH61} as a simplified version of the celebrated Hodgkin-Huxley model \cite{HH52}. The latter is proposed by Alan Hodgkin and Andrew Huxley in 1952 and led them to receive a ``Nobel Prize in Physiology or Medicine'' in 1963. The model gives a precise mathematical description of how the spikes in neuron cells are generated and propagated in time.

An important property of the kinetic FitzHugh-Nagumo equation is that the dynamics may exhibit non-equilibrium stationary states. The equation may have periodic solutions or converge towards  a unique stationary state depending on the strength of the connectivity of the neural network. We are interested in the latter case which takes place in a weak connectivity regime. 

Moreover, the kinetic FitzHugh-Nagumo equation is known to be hypoelliptic, see e.g. \cite{MQT16}. This can be seen from the internal dynamics. The noise generated in the system acts as a random input only for the voltage variable $v$ and does not change the adaptation variable $x$. 

The evolution of the probability density of finding neurons at time $t \geq 0$ with an adaptation variable $x$ and a voltage variable $v$ such that $(x,v) \in \O \times \mc V = \R \times \R$ is given by
\begin{align} \label{eq:FH}
	\p_t f =  \mc L [f] := \p_x  (A(x,v)f ) + \p_v  (B (x,v) f) + \p_{vv} f,
\end{align} where 
\begin{align} \label{potentials}
	A (x,v) = a x - b v, \quad \mbox{and} \quad B(x,v) = x + v (v-1) (v-c),
\end{align} with some positive constants $a,b,c$. %We refer to, e.g., Appendix A of \cite{MQT16} for the derivation from a stochastic particle system and the well-posedness of the kinetic FitzHugh-Nagumo equation. 

In \cite{MQT16}, the authors studied the existence and uniqueness of solutions to this equation. Moreover, they showed the existence of a nontrivial stationary state and non-quantitative exponential convergence towards the stationary state by using a generalised version of the Kre\u{\i}n-Rutman theorem. 
In \cite{C21}, the author improved this result and provided a quantitative exponential convergence to equilibrium by using Harris's theorem. The main result of \cite{C21} is given by

\begin{theorem} \label{thm:FitzHugh}
	Suppose that $t \mapsto f_t$ is the solution of the kinetic FitzHugh-Nagumo equation \eqref{eq:FH}-\eqref{potentials} with the initial data $f_0 \in \mc P_\phi (\R \times \R)$. Then the kinetic FitzHugh-Nagumo equation admits a unique stationary state $f_\infty$ and there exist positive constants $C, \lambda$ which are independent from $f_0$, such that, for all $t \geq 0$ and
	\begin{align} \label{eq:FNconv}
		\|f_t - f_\infty \|_{\phi} \leq C e^{-\lambda t} \|f_0 - f_\infty\|_{\phi},
	\end{align} with the weight function 
	\begin{align*}
		\phi (x, v) = e^{\chi (|x|^2 + |v^2|)}.
	\end{align*} where $\chi>0$ is a positive constant.
\end{theorem}
\begin{remark}
	We note that in \cite{C21}, the author proves more general version of \eqref{eq:FNconv} in the norm $\|\cdot \|_{L_\phi^q}$ when $f_0 \in  L_\phi^p (\R\times \R)$, $p \in [1, +\infty)$ and $q \in [1,p]$. We state an adapted version of the result only in the weighted $L^1$ space with the weight function $\phi$ (corresponds to Section 5 in \cite{C21}) by taking a probability measure initial data. 
\end{remark}
\begin{proof}[Idea of the proof of Theorem \ref{thm:FitzHugh}] The proof uses Harris's theorem. 
	The \ref{Minorisation} relies on proving that $f$ is bounded from below by a positive constant at an initial arbitrary point $z_0 = (x_0,v_0)$ and extending this positive lower bound to a small ball around this point. This part is shown via ensuring the local continuity of $f$ by a regularisation argument which can be found in Section 3 of \cite{C21}. Then the \ref{Minorisation} is an immediate result of the \emph{spreading of positivity lemma} which says 
	\begin{align*}
		\mbox{If $f \geq \ve>0$ in $[0,t) \times B(z_0, R)$, then $f \geq C \ve$ in $[t/2, t] \times B(z_0, \alpha R)$ for some $\alpha >1$ and $C>0$.}
	\end{align*}
	For the \ref{Foster_Lyap} we follow a similar strategy as the previous sections. We apply a change of variable $v':= bv$ and re-write $\mc L$ in \eqref{eq:FH} as
	\begin{align*}
		\mc L  [f] = \p_x (A(z)f ) + \p_v (B (z) f) + \frac{1}{b^2} \p_{vv} f,
	\end{align*} where the formal adjoint of $\mc L$ is given by
	\begin{align*}
		\mc L^* \phi = - A(z) \p_x \phi - B(z) \p_v \phi  + \frac{1}{b^2} \p_{vv} \phi. 
	\end{align*} Then  for $\varphi (x,v) = e^{\frac{\gamma}{2} (|x|^2 + |v|^2)}$ 
	we obtain (assuming that $\gamma^2 > bc$),
	\begin{align*}
		\frac{	\mc L^* \varphi}{\varphi}  = - a \gamma|x|^2 - \frac{1}{b^3} |v|^4 + \frac{c+1}{b^2} v^3 + \frac{\gamma^2 -bc}{b^2} |v|^2 + \frac{\gamma}{b^2}
		\leq - a \gamma|x|^2 - \frac{1}{b^3} |v|^4 + C_1 v^3 + C_2 |v|^2 + C_3,
	\end{align*} where $C_1, C_2, C_3$ are some positive constants. This yields
	\begin{align*}
		\mc L^* \phi \leq - \zeta \phi + D,
	\end{align*} for some $\zeta, D >0$. 
	Thus the result follows from Theorem \ref{thm:harris}.
\end{proof}

\section{Discussion and perspectives}

\label{sec:Discussion}

In this review, we presented Harris-type theorems and how to use them in the PDE context to obtain results on trend to equilibrium for kinetic equations arising in physics and biology. The techniques are well-adapted for equations exhibiting non-equilibrium, non-explicit stationary states, initial data with bad local regularity (e.g. Dirac measure) and slowly decaying tails and lastly to obtain quantifiable convergence rates. These methods provide an alternative to classical hypocoercivity techniques, general relative entropy methods, and non-quantitative methods like the Kre\u{\i}n-Rutman theorem. They allow us to obtain spectral gap results in $L^1$ (or total variation) setting under relaxed assumptions. 

On the other hand, we point out that these methods apply only to linear kinetic equations and the quantitative rates are often not optimal. 

However, we believe that they provide promising results in the cases including but not limited to
\begin{itemize}
	\item \textbf{Absence of spectral gap.} The results on the long-time behaviour of kinetic equations in the cases where convergence is sub-geometric are rare compared to the cases where there is a spectral gap. One reason is that these cases are often more difficult to study mathematically as there is no general methodology to tackle them with. However, Harris-type theorems provide an insightful method to obtain quantifiable subgeometric rates of convergence which are optimal in some cases (see e.g. Sec. \ref{sec:Knudsen}). Especially, the application of the theorems is straightforward once the hypotheses are proven.
	\item \textbf{Nonlinearities.} Harris-type theorems do not apply directly to nonlinear equations. However, they can be used both for proving the existence of stationary states to nonlinear equations, as well as for the perturbative analysis of the nonlinear problems close to equilibrium (see e.g. Sec. \ref{sec:linearBGK} and Sec. \ref{sec:runtumble}). 
\end{itemize}

\section*{Acknowledgements} 
The author would like to thank the anonymous reviewer whose careful reading and comments significantly improved the presentation of the paper. This work was partially supported by the Vienna Science and Technology Fund (WWTF) with a Vienna Research Groups for Young Investigators project, grant VRG17-014. The author would like to thank the Isaac Newton Institute for Mathematical Sciences for support and hospitality during the programme ``Frontiers in kinetic theory: connecting microscopic to macroscopic scales - KineCon 2022'' when work on this paper was undertaken. This work was supported by EPSRC Grant Number EP/R014604/1.

%\section*{Data Availability Statement}
%Data sharing is not applicable to this article as no new data were created or analysed in this study. 
%
%\section*{Author Declarations}
%The author has no conflicts to disclose.

\bibliography{Harris}
\end{document}